\documentclass{amsart}
\usepackage{amscd,amssymb,amsopn,amsmath,amsthm,graphics,amsfonts,accents,enumerate,verbatim,calc}
\usepackage[dvips]{graphicx}
\usepackage[colorlinks=true,linkcolor=red,citecolor=blue]{hyperref}
\usepackage[all]{xy}

\usepackage{tikz}
\usepackage{mathrsfs}
\usepackage{tikz-cd}

\usepackage{cite}
\usepackage{tikz}

\calclayout

\newcommand{\rt}{\rightarrow}
\newcommand{\lrt}{\longrightarrow}

\newcommand{\st}{\stackrel}

\newcommand{\la}{\lambda}
\newcommand{\La}{\Lambda}

\newcommand{\lan}{\langle}
\newcommand{\ran}{\rangle}

\newcommand{\CA}{\mathcal{A} }
\newcommand{\CC}{\mathcal{C} }

\newcommand{\CG}{\mathcal{G} }

\newcommand{\CM}{\mathcal{M} }

\newcommand{\CP}{\mathcal{P} }
\newcommand{\CQ}{\mathcal{Q} }
\newcommand{\CR}{\mathcal{R} }
\newcommand{\CS}{\mathcal{S} }
\newcommand{\CT}{\mathcal{T} }
\newcommand{\CX}{\mathcal{X} }

\newcommand{\CV}{\mathcal{V}}
\newcommand{\CU}{\mathcal{U}}

\newcommand{\I}{{\lan I \ran}}

\newcommand{\Mod}{{\rm{Mod\mbox{-}}}}

\newcommand{\mmod}{{\rm{{mod\mbox{-}}}}}

\newcommand{\prj}{{\rm{prj}\mbox{-}}}

\newcommand{\add}{{\rm{add}\mbox{-}}}

\newcommand{\Ker}{{\rm{Ker}}}

\newcommand{\rad}{{\rm{rad}}}

\newcommand{\Hom}{{\rm{Hom}}}
\newcommand{\Ext}{{\rm{Ext}}}
\newcommand{\End}{{\rm{End}}}

\theoremstyle{plain}
\newtheorem{theorem}{Theorem}[section]
\newtheorem{corollary}[theorem]{Corollary}
\newtheorem{lemma}[theorem]{Lemma}

\newtheorem{proposition}[theorem]{Proposition}

\theoremstyle{definition}
\newtheorem{definition}[theorem]{Definition}
\newtheorem{example}[theorem]{Example}

\newtheorem{construction}[theorem]{Construction}

\newtheorem{remark}[theorem]{Remark}

\theoremstyle{plain}

\theoremstyle{definition}

\numberwithin{equation}{section}

\begin{document}

\title[When stable Cohen-Macaulay  Auslander algebra is semisimple]{When stable Cohen-Macaulay  Auslander algebra is semisimple}
\dedicatory{Dedicated to people in Afghanistan}
\author[Rasool Hafezi  ]{Rasool Hafezi }

\address{Department of Mathematics, University of Isfahan, P.O.Box: 81746-73441, Isfahan, Iran }
\email{hafezi@ipm.ir, hafezira@gmail.com}

\thanks{}


\begin{abstract} Let $\text{Gprj}\mbox{-}\La$ denote the category of Gorenstein projective modules over an Artin algebra $\La$ and the category $\mmod (\underline{\text{Gprj}}\mbox{-}\La)$  of finitely presented functors over the stable category $\underline{\text{Gprj}}\mbox{-}\La$. In this paper, we study those algebras $\La$ with $\mmod (\underline{\text{Gprj}}\mbox{-}\La)$ to be a semisimple abelian category, and called $\Omega_{\CG}$-algebras. The class of $\Omega_{\CG}$-algebras contains important classes of algebras, including gentle algebras. Over an $\Omega_{\CG}$-algebra $\La$,  the structure of the almost split sequences in the morphism category $\text{H}(\text{Gprj}\mbox{-}\La)$ and the monomorphism category $\CS(\text{Gprj}\mbox{-}\La)$ of $\text{Gprj}\mbox{-}\La$  is investigated. Among other applications, we provide some results for the Cohen-Macaulay Auslander algebras of $\Omega_{\CG}$-algebras.

\end{abstract}

\maketitle
\section{Introduction and preliminaries}\label{Section 1}
Let us begin with some notation.

Let $\mathcal{X}$ be an additive category. The Hom sets will be denoted either by $\Hom_{\CX}(-, -)$, $\CX(-, -)$ or even just $(-,-)$, if there is no risk of ambiguity. By definition, a (right) $\CX$-module is a contravariant additive functor $F:\CX \rt \CA \rm{b}$, where
$\CA \text{b}$ denotes the category of abelian groups. The $\CX$-modules and natural transformations between them, called morphisms, form an abelian category denoted by $\Mod \CX$. An $\CX$-module $F$ is called finitely presented if there exists an exact
sequence 
$$\CX(-, X) \rt \CX(-, X')\rt F\rt 0,$$
with $X$ and $X'$ in $\CX$. All finitely presented $\CX$-modules form a full subcategory of $\Mod \CX$, denoted by $\mmod \CX$.

Throughout the paper  $\La$  denotes  an Artin algebra and $\mmod \La$ the category of finitely generated right $\La$-modules. From now on to the end of the introduction, let $\CX$ be a full subcategory of $\mmod \La$. Assume   $M \in \mmod \La$.  A right $\CX$-approximation of $M$ is a morphism $f: X \rt M$ such that $X\in \CX$ and any other morphism $X' \rt M$ with $X'\in \CX$ factors through $f$. $\CX$ is called a contravariantly finite subcategory of $\mmod \La$ if every $M \in \mmod \La$ admits a right $\CX$-approximation. Left $\CX$-approximations and covariantly finite subcategories are defined dually. If $\CX$ is both a contravariantly finite and a covariantly finite subcategory of $\mmod \La$, then it is called a functorially finite subcategory.

If $\CX$ is contravariantly finite, then any morphism in $\CX$ admits a weak kernel. Then it is know that $\mmod \CX$ is abelian, see \cite[Chapter III, Section 2]{Au}.

Let $\CX \subseteq \mmod \La$ contain the subcategory $\text{prj}\mbox{-}\La$ of projective modules in $\mmod \La$.  We consider the stable category of $\mathcal{X}$, denoted by $\underline{\mathcal{X}}.$ The objects of $\underline{\mathcal{X}}$ are the same as the objects of $\mathcal{X}$, which we usually denote by  $\underline{X}$ when an object $X \in \mathcal{X}$ considered as an object in the stable category, and  the morphisms are given by $\underline{\text{Hom}}_{\La}(\underline{X}, \underline{Y})= \Hom_{\La}(X, Y)/ \CP(X, Y)$, where $\CP(X, Y)$ is the subgroup of $\Hom_{\La}(X, Y)$  consisting of those morphisms from $X $ to $Y$ which factor through a projective module  in $\mmod \La.$ We also denote by $\underline{f}$ the residue class of $f: X \rt Y$ in $\underline{\text{Hom}}_{\La}(\underline{X}, \underline{Y})$. It is well-known that the canonical functor $\pi: \mathcal{X}\rt \underline{\mathcal{X}}$ induces a fully faithful functor  functor $\pi^*:\mmod \underline{\mathcal{X}}\rt \mmod \mathcal{X}$. Hence due to this embedding we can  identify the functors in $\mmod \underline{\mathcal{X}}$ as functors in $\mmod \mathcal{X}$ vanish on the projective modules.

By our embedding, one can see that $\mmod \underline{\mathcal{X}}$ is  a Serre subcategory, i.e., it is closed under taking subobjects, quotients and extensions.  The Gabriel quotient category $\frac{\mmod \CX}{\mmod \underline{\CX}}$ is by definition the localization of  $\mmod \CX$ with respect to
the collection of all morphisms whose kernels and cokernels are in $\mmod \underline{\CX}$, see \cite{Ga} for more details. 

It is proved in \cite[Theorem 3.5]{AHK} that if $\text{prj}\mbox{-}\La \subseteq \CX $ is a contravariantly finite subcategory, then there exists an equivalence 

$$\frac{\mmod \CX}{\mmod \underline{\CX}}\simeq \mmod \La.$$

The above equivalence  is a relative version of an equivalence given by Auslander 50 years ago, and now is called the Auslander's formula. A  considerable part of Auslander's work on the representation
theory of finite dimensional, or more general Artin, algebras is motovated by  this
formula. Auslander's  formula suggests that for studying  $\mmod \La$ one may
study $\mmod (\mmod \La)$ that has nicer homological properties than $\mmod \La$, and then translate the results back to $\mmod \La$.

Inspired by the influence of the Auslander's formula on the modern representation theory, and having in hand
 a relative version of the formula, as presented in the above, we are interested in studying $\mmod \CX$ and $\mmod \underline{\CX}$ for some certain subcategories. One of the important subcategories in the representation theory is the subcategory of Gorenstein projective modules, which plays a key role in Gorenstein homology algebra. In this paper we focus on this kind of subcategories.

Let us in below recall the notion of Gorenstein projective modules and relevant material.

 A complex
\[P^\bullet:\cdots\rightarrow P^{-1}\xrightarrow{d^{-1}} P^0\xrightarrow{d^0}P^1\rightarrow \cdots\]
in $\text{prj}\mbox{-}\La$ is called a totally acyclic complex if it is acyclic and the induced Hom complex $\Hom_{\La}(P^\bullet, \La)$ is also acyclic. A module $G$ in $\mmod \La$ is called {\it Gorenstein projective} if it is isomorphic to a syzygy of a totally acyclic complex $P^\bullet$ \cite{AB, EJ}. We denote by $\text{Gprj}\mbox{-}\La$ the full subcategory of $\mmod \La$ consisting of all Gorenstein projective modules. By definition is is plain $\text{Gprj}\mbox{-}\La$ contains $\text{prj}\mbox{-}\La$.

 An Artin algebra $\La$ is called of finite Cohen-Macaulay type, or  CM-finite for short, if there are only finitely many isomorphism classes of indecomposable finitely generated Gorenstein projective modules \cite{B}.

Clearly, $\La$ is a CM-finite algebra  if and only if there is a finitely generated module $G$ such that $\text{Gprj}\mbox{-}\La=\add(G)$, where $\add(G)$ is the additive subcategory of $\mmod \La$ consisting of all modules isomorphic to a direct summand of a finite direct sum of copies of $G$. In this case $G$ is called an additive generator of $\text{Gprj}\mbox{-}\La$ and, in addition, if $G$ is basic, then  the Artin algebra $\End_{\La}(G)$  is called the Cohen-Macaulay Auslander algebra of $\La.$ Moreover, the stable Cohen-Macaulay Auslander algebra of $\La$ is defined as the endomorphism algebra $\underline{\text{End}}_{\La}(\underline{G})$ of $\underline{G}$ in the stable category $\underline{\text{Gprj}}\mbox{-}\La.$ If $\La$ is  a  CM-finite algebra
with G as an additive  generator, then the evaluation functor
$$\zeta_G: \mmod (\text{Gprj}\mbox{-}\La)\rt \mmod \text{End}_{\La}(G)$$
defined by $\zeta_G(F)=F(G)$ is an equivalence of categories \cite[Proposition 2.7(c)]{Au2}. Further, the restriction of $\zeta_G$ to  $\mmod (\underline{\text{Gprj}}\mbox{-}\La)$ induces the equivalence $\mmod (\underline{\text{Gprj}}\mbox{-}\La)\simeq \mmod\underline{\text{End}}_{\La}(\underline{G}).  $

  We say that $\La$ is an $n$-Iwanaga-Gorenstein, or simply $n$-Gorenstein, algebra if the injective dimension of $\La$ both as a left and a right $\La$-module is at most $n$. We often omit ``$n$" when the value is not important to know. 
  It is known that if $\La$ is Gorenstein, then  $\text{Gprj}\mbox{-}\La $ is contravariantly finite in $\mmod \La$, see e.g. \cite[Proposition 4.7]{B}. In particular, over Gorensetin algeras $\mmod (\text{Gprj}\mbox{-}\La)$ is abelian.

The simplest cases of studying $\mmod (\underline{\text{Gprj}}\mbox{-}\La)$, at least in the homological dimensions sense, is when the global projective dimension of $\mmod (\underline{\text{Gprj}}\mbox{-}\La)$ is zero, or a semisimple abelian category, i.e., any object is projective. We call an algebra with this property $\Omega_{\CG}$-algebra; Some basic properties of them will be studied in Section \ref{Omega-algebras}. Especially we show in Proposition \ref{CM-finite} any $\Omega_{\CG}$-algebra is CM-finite. Hence, an $\Omega_{\CG}$-algebra $\La$ is nothing else  to say that the associated stable Cohen-Macaulay algebra of $\La$ is a semisimple algebra. Such an observation was investigated in the remarkable paper of Auslander and Reiten, where they showed that  for a dualizing $R$-variety $\mathcal{C}$: the global dimension of  $ \mmod (\underline{\text{mod}}\mbox{-}\CC)$ is zero if and only if $\mmod \CC$ is Nakayama and whose lowey length is at most two \cite[Theorem 10.7]{AR1}. The common case with our study is when $\mathcal{C}=\text{prj}\mbox{-}\La$ and $\La$ is a self-injective algebra.

The class of $\Omega_{\CG}$-algebras contains important class of algebras in the representation theory such as the class of gentle algebras, and more general the class of  quadratic monomial algebras (Remark \ref{Remark 2.8}). Gentel algebras  can be found in many places of mathematics. The study of gentle algebras was initiated
by Assem and Skowro\'{n}ski \cite{ASk}  in order to study iterated tilted algebras of type $\mathbb{A}$. Viewing gentle algebras as $\Omega_{\CG}$-algebras provide them a functorial approach. Even though the provided functorial approach is not completely determined gentle algebras, however it might be helpful to consider them in the larger class of $\Omega_{\CG}$-algebras. 
 
 We   are  also interested in studying parallel of $\mmod (\text{Gprj}\mbox{-}\La)$. The $\Omega_{\CG}$-algebras as the simplest case which  we choose to study $\mmod (\underline{\text{Gprj}}\mbox{-}\La)$ also affect nicely to the representation theory of the corresponding    Cohen-Macaulay Auslander  algebras as our results will indicate.
To study $\mmod (\text{Gprj}\mbox{-}\La)$ and  $\mmod (\underline{\text{Gprj}}\mbox{-}\La)$, respectively,  we apply the morphism category $\text{H}(\text{Gprj}\mbox{-}\La)$  and the monomorphism category $\CS(\text{Gprj}\mbox{-}\La)$ of  $\text{Gprj}\mbox{-}\La$.

 Let us first recall the aforementioned categories as follows: the morphism category of $\text{H}(\mmod \La)$, for short $\text{H}(\La)$, has as objects the $\La$-homomorphisms in $\mmod \La$, and morphisms are given by  commutative diagrams. In fact, it is equivalent to the category of finitely
 generated modules over the lower  triangular matrix ring   $T_2(\La)= \tiny {\left[\begin{array}{ll} \La & 0 \\ \La & \La \end{array} \right]}$. 
We denote an object in $\text{H}(\La)$ by $\left(\begin{smallmatrix} A \\ B\end{smallmatrix}\right)_f$, where $f:A\rt B$ a map in $\mmod \La$, and a morphism in $\text{H}(\La)$ between $\left(\begin{smallmatrix} A \\ B\end{smallmatrix}\right)_f$ and $\left(\begin{smallmatrix} C \\ D\end{smallmatrix}\right)_g$ is denoted by $\left(\begin{smallmatrix} \alpha \\ \beta\end{smallmatrix}\right) $, where $\alpha:A\rt C$ and $\beta:B\rt D$ in $\mmod \La$ such that $\beta f=g \alpha.$ Let $\CS(\mmod \La)$ (or for short $\CS(\La)$) denote the subcategory of all  monomorphisms in $\mmod \La.$ For simplicity, especially in the diagrams or figures we write $AB_f$ in stead of $\left(\begin{smallmatrix} A \\ B\end{smallmatrix}\right)_f$, moreover if the morphism $f$ is clear from the context we only write $AB$, especially for the case $f$ is either the identity morphism or zero map. The  category  $\CS(\La)$ has  been recently  studied extensively by Ringel and Schmidmeie \cite{RS}. Now the morphism (resp. monomorphism)  category of $\text{H}(\text{Gprj}\mbox{-}\La)$ (resp. $\CS(\text{Gprj}\mbox{-}\La)$)  is the subcategory of $\text{H}(\La)$ (resp. $\CS(\La)$) consisting of all objects $\left(\begin{smallmatrix} A \\ B\end{smallmatrix}\right)_f$ such that $A, B \in \text{Gprj}\mbox{-}\La$ (resp. $A, B$ and $\text{Cok} (f) \in \text{Gprj}\mbox{-}\La$). Under the equivalence $\text{H}(\La)\simeq \mmod T_2(\La)$, the subcategory  $\CS(\text{Gprj}\mbox{-}\La)$ is mapped into the subcategory $\text{Gprj}\mbox{-}T_2(\La)$ of Gorenstien projective modules over $T_2(\La)$.

The reason of applying the categories $\text{H}(\text{Gprj}\mbox{-}\La)$ and $\CS(\text{Gprj}\mbox{-}\La)$ is based on the  relationship given by the following  functors:
$$\Phi: \text{H}(\text{Gprj}\mbox{-}\La)\rt \mmod (\text{Gprj}\mbox{-}\La),  \ \ \left(\begin{smallmatrix} A \\ B\end{smallmatrix}\right)_f \mapsto \text{Cok}((-, A)\st{(-, f)}\lrt (-, B)),$$
where we apply the Yoneda  functor on the morphism $f:A\rt B$ and then take cokernel of the obtained morphism $(-, f):(-, A)\rt (-, B)$ in $\mmod  (\text{Gprj}\mbox{-}\La),$ and 

$$\Psi:\CS(\text{Gprj}\mbox{-}\La) \rt \mmod (\underline{\text{Gprj}}\mbox{-}\La) ,  \ \ \left(\begin{smallmatrix} A \\ B\end{smallmatrix}\right)_f \mapsto \text{Cok}((-, B)\st{(-, p)}\lrt (-, \text{Cok}(f))), $$
where $p:B\rt \text{Cok}(f)$ is the canonical quotient map.

The functor $\Psi$ was studied in \cite{E} and \cite{RZ} for the case that $\mmod \La$ over  certain algebras $\La$ are involved in place of  the subcategory  $\text{Gprj}\mbox{-}\La$, indeed, we have the functor $\Psi$ from $\CS(\La)$ to $\mmod (\underline{\text{mod}}\mbox{-}\La)$. We also add here the origin of the functor $\Psi$ goes back to \cite{AR2}.  Respect to fairly large class of subcategories, including $\text{Gprj}\mbox{-}\La$, the functor $\Psi$ is treated by the author in \cite{H}. It is proved that  the functor $\Psi$ is dense, full and objective; hence, by \cite[Appendix]{RZ}, the functor induces the equivalence 
$$\CS(\text{Gprj}\mbox{-}\La)/\CU\simeq \mmod (\underline{\text{Gprj}}\mbox{-}\La), $$ 
where $\CU$ is the ideal of $\CS(\text{Gprj}\mbox{-}\La)$ generated by objects of the forms $\left(\begin{smallmatrix} G \\ G\end{smallmatrix}\right)_1$ or $\left(\begin{smallmatrix} 0 \\ G\end{smallmatrix}\right)_0$, where $G$ runs through $\text{Gprj}\mbox{-}\La.$ The quotient category $\CS(\text{Gprj}\mbox{-}\La)/\CU$ has the same objects as $\CS(\text{Gprj}\mbox{-}\La)$. For objects $X$ and $Y$, let $I(X, Y)$ denotes the morphisms of $\Hom_{\text{H}(\La)}(X, Y)$ that factor through a finite  direct sum of objects of the forms $\left(\begin{smallmatrix} G \\ G\end{smallmatrix}\right)_1$ or $\left(\begin{smallmatrix} 0 \\ G\end{smallmatrix}\right)_0$ where $G \in \text{Gprj}\mbox{-}\La.$ The morphisms spaces of $\CS(\text{Gprj}\mbox{-}\La)/\CU$
are defined as the quotients $\CS(\text{Gprj}\mbox{-}\La)/\CU(X, Y):=\Hom_{\text{H}(\La)}(X, Y)/I(X, Y)$.

Similar observation also holds for the functor $\Phi$. This functor was already studied in \cite{AR2}, of course for the case that we are dealing with $\mmod \La$ in stead of $\text{Gprj}\mbox{-}\La.$ In the same manner of $\Psi$ one can prove that the functor $\Phi$ is full, dense and objective, see also \cite[Proposition 2.5]{E}. Hence, again by \cite[Appendix]{RZ}, we have the equivalence 

 $$\text{H}(\text{Gprj}\mbox{-}\La)/\CV\simeq \mmod (\text{Gprj}\mbox{-}\La) $$ 
 where $\CV$ is the ideal of $\CS(\text{Gprj}\mbox{-}\La)$ generated by objects of the forms $\left(\begin{smallmatrix} G \\ G\end{smallmatrix}\right)_1$ or $\left(\begin{smallmatrix} G \\ 0\end{smallmatrix}\right)_0$, where $G$ runs through $\text{Gprj}\mbox{-}\La.$ 

Therefore, the above equivalences make  clear  why we choose  $\text{H}(\text{Gprj}\mbox{-}\La)$ and $\CS(\text{Gprj}\mbox{-}\La)$ to study $\mmod (\text{Gprj}\mbox{-}\La)$ and  $\mmod (\underline{\text{Gprj}}\mbox{-}\La)$, respectively. In spite of inducing the equivalence via the functor $\Psi$ and $\Phi$, they also behave well with respect to the Auslander-Reiten theory. More precisely, it is proved in \cite[Section 5]{H} that apart from certain almost split sequences in $\text{H}(\text{Gprj}\mbox{-}\La)$ the functor $\Psi$ carries all other   almost split sequences in $\text{H}(\text{Gprj}\mbox{-}\La)$  to the ones in $\mmod (\underline{\text{Gprj}}\mbox{-}\La)$. An analogous result is valid for the functor $\Phi$. For this case, we refer to \cite[Section 5]{HMa}. Although, in the latter paper the functor $\Phi$ is defined for the morphism category $\text{H}(\CA)$ of a sufficiently nice abelian category $\CA$. But the proof still
works in our setting. However, for our use in this paper we sketch a proof for the cases in place we are dealing with. 

The paper is organized and illustrated in more details as follows. In Section \ref{Omega-algebras}, the notion of $\Omega_{\CG}$-algebras is defined. Some basic properties and examples of $\Omega_{\CG}$-algebras will be discussed. A classification of indecomposable non-projective Gorenstein projective modules over $\Omega_{\CG}$-algebras will be given that is a generalization of the ones  in \cite{CSZ} for quadratic monomial algebras and \cite{Ka} for gentle algebras to $\Omega_{\CG}$-algebras (see Remark \ref{remark 2.7} for details). The classification is given by defining an equivalence relation on the indecomposable non-projective modules in $\text{Gprj}\mbox{-}\La$.   We shall also  describe the singularity category over Gorenstein $\Omega_{\CG}$-algebras (Proposition \ref{proposition 2.8}).  In Section \ref{Section 3}, for an $\Omega_{\CG}$-algebra $\La$, the stable Auslander-Reiten quiver of $\CS(\text{Gprj}\mbox{-}\La)$ is completely determined in the main theorem of this section (Theorem \ref{Theorem 4.10}). In particular, we will establish a bijection between of the components of the stable Auslander-Reiten quiver and the equivalence classes of the relation already defined. In addition, if $\La$ is assumed further to be a finite dimensional algebra over an algebraic closed filed $k$ of  	characteristic different from two,  then we will show that the stable category of Gorenstein projective $T_2(\La)$-modules is still described well. In Section \ref{Section 4}, we will compute the almost split sequences in $\text{H}(\text{Gprj}\mbox{-}\La)$ over $\Omega_{\CG}$-algebras with certain ending terms. As an interesting application in Corollary \ref{Cor 4.4}  a connection between certain simple modules over the  Cohen-Macaulay Auslander algebra of an $\Omega_{\CG}$-algebra is given. In the last section, we will study those $\Omega_{\CG}$-algebras that are 1-Gorenstein. In Proposition \ref{RadGorp}, we will an equivalent condition of such algebras via their radicals. We will show over a 1-Gorenstein $\Omega_{\CG}$-algebras $\La$: the Cohen-Macaulay Auslander algebra  of  $\La$ is representation-finite if and only if  so is $\La$.  Finally,  in Proposition \ref{linearquiver} we observe that $\Omega_{\CG}$-algebras (not necessary 1-Gorenstein)  are a good source to produce CM-finite algebras.

\section{$\Omega_{\CG}$-algebras} \label{Omega-algebras}
In this section, we will introduce the notion of $\Omega_{\CG}$-algebras. Some basics properties and examples of $\Omega_{\CG}$-algebras will be provided. In particular, we characterize the isoclasses of indecomposable non-projective Gorenstein-projective modules over an $\Omega_{\CG}$-algebra in terms of an equivalence relation defined on $\text{Gprj}\mbox{-}\La$. Next we will apply this characterization to describe the singularity categories over Gorenstein $\Omega_{\CG}$-algebras.

Throughout this section we assume that   $\rm{Gprj} \mbox{-} \La$ is contravariantly finite in $\mmod \La.$ Recall that a subcategory $\CX$ of  $\mmod \La$ is resolving if it contains all projectives, and  closed under extensions and the kernels of epimorphisms. On the other hand, since  $\rm{Gprj} \mbox{-} \La$ is resolving it also becomes a functorially finite subcategory of $\mmod \La$ (due to \cite{KS}). So, by \cite{AS}, we can talk about the Auslander-Reiten theory in $\rm{Gprj} \mbox{-} \La.$

By definition of the Gorenstein projective modules, we can see for each $G \in \rm{Gprj} \mbox{-} \La$, $G^*=\Hom_{\La}(G, \La)$ is a  Gorenstein projective module in $\mmod \La^{\rm{op}}.$ Then the duality $(-)^*:\prj \La \rt \prj \La^{\rm{op}}$ can be naturally generalized to the  duality $(-)^*:\rm{Gprj} \mbox{-} \La \rt \rm{Gprj} \mbox{-} \La^{\rm{op}}.$

Motivated by Section 5 of \cite{H1} (an unpublished paper by the author) we have the following definition.

\begin{definition}
	Let $\La$ be an Artin algebra. Then $\La$ is called $\Omega_{\CG}$-algebra if $\rm{mod}\mbox{-}(\underline{\rm{Gprj}}\mbox{-} \La)$ is a  semisimple abelian category, i.e. all whose  objects are projective.
	
\end{definition}
In the sequel,  we give some basic properties of $\Omega_{\CG}$-algebras.

 We start with following lemmas.
 
 \begin{lemma}\label{Radical}
 	Let $\La$ be an Artin algebra.  Let $P$ be an indecomposable projective in $\mmod \La.$
 	\begin{itemize}
 		\item[$(i)$] If $f: G  \rt \rm{rad}(P)$ is a  minimal right $\rm{Gprj} \mbox{-} \La$-approximation. Then $i \circ f$ is a minimal  right  almost split morphism in $\rm{Gprj} \mbox{-} \La$. Here, $i:\rm{rad}(P)\hookrightarrow P$ denotes  the canonical inclusion.
 		\item[$(ii)$] If $g: G'  \rt \rm{rad}(P^*)$ is a  minimal right $\rm{Gprj} \mbox{-} \La^{\rm{op}}$-approximation. Then $(j \circ g)^*:  P \rt (G')^*$ is a minimal left almost split morphism in $\rm{Gprj} \mbox{-} \La$. Here, $j:\rm{rad}(P^*)\hookrightarrow P^*$ denotes
 		the canonical inclusion, and by abuse of notation we identify $P^{**}$ and $P$ because of the natural isomorphism $P^{**}\simeq P$.
 	\end{itemize}
 \end{lemma}
 \begin{proof}
 	$(i)$ We only need to prove that for any non-split epimorphism $h:G_0 \rt P$, there exists morphism $l:G_0 \rt G$	so that $i\circ f\circ l=h.$ Since $h$ is a non-split epimorphism, then $\rm{Im}( h)$ is a proper submodule in $P$, i.e. $\rm{Im}(h) \subseteq \rm{rad}(P).$   So we can write $h=i \circ s$ for some $s: G_0 \rt \rm{rad}(P)$. Now since $f$ is a right    $\rm{Gprj} \mbox{-} \La$-approximation, then there is $l:G_0 \rt G$ such that $f \circ l=s$ and clearly it works for what we need.
 	
 	$(ii)$ follows from $(i)$ and the duality $(-)^*:\rm{Gprj} \mbox{-} \La \rt \rm{Gprj} \mbox{-} \La^{\rm{op}}.$ 
 \end{proof}

\begin{lemma}\label{Lemma1}
	Let $\La$ be an $\Omega_{\CG}$-algebra. Then any short exact sequence in $\rm{Gprj}\mbox{-} \La$ is a direct sum of the short exact sequences of the  form  $0 \rt \Omega_{\La}(A) \rt P \rt A \rt 0$, $0 \rt 0 \rt B  \st{1} \rt B \rt 0$ and $0 \rt C \st{1} \rt C \rt 0 \rt 0$ 	for some $A, B $ and $C$ in $\rm{Gprj} \mbox{-} \La.$  
\end{lemma} 
\begin{proof}
	Take an exact sequence $0 \rt G_2 \rt G_1 \rt G_0 \rt 0$ in $\rm{Gprj}\mbox{-} \La$. This short exact sequence induces the sequence  $$ 0 \rt (-, G_2) \rt (-, G_1) \rt (-, G_0) \rt F \rt 0 \ \ \ \  (1)$$
	in $\mmod \rm{Gprj}\mbox{-}\La.$ Since $\rm{mod}\mbox{-}(\underline{\rm{Gprj}}\mbox{-} \La)$ is a  semisimple abelian category then $F \simeq (-, \underline{G})$ for some $G$ in $\rm{Gprj}\mbox{-} \La$. On the other hand, we know a minimal projective resolution of $(-, \underline{G})$ is as the following $$ 0 \rt (-, \Omega_{\La}(G)) \rt (-, P) \rt (-, G) \rt (-, \underline{G}) \rt 0 \ \ \ (2) $$ where $P \rt G $ is a projective cover of $G.$ By comparing $(1)$ and $(2)$ as two projective resolutions of $F$ in $\mmod (\rm{Gprj} \mbox{-} \La)$, and using this fact that the latter is minimal  then we get our result. Indeed, the fact we have used here  is  known  in the homology algebra, that is,  any projective resolution of an object in a perfect category  is a direct sum of  the minimal projective resolution of the given object and possibly some split exact complexes. 
\end{proof}

\begin{proposition}\label{CM-finite}
	Let $\La$ be an  $\Omega_{\CG}$-algebra.  Then $\La$ is $\rm{CM}$-finite.
\end{proposition}
\begin{proof}
 Let $G$ be an indecomposable non-projective Gorenstein projective  module. According to Lemma \ref{Lemma1}, since an  almost split sequence does not split, so the almost split sequence in $\rm{Gprj}\mbox{-}\La$ obtained by getting the projective cover of $G$. Thus it is of the form  $0 \rt \Omega_{\La}(G) \rt P \st{f} \rt G \rt 0$ in $\rm{Gprj} \mbox{-} \La$. Let $Q$ be an indecomposable direct summand of $P$. The natural injection $f_{\mid Q}:Q \rt G$ is  irreducible by using the  general facts in the theory of almost split sequences for subcategories, e.g. \cite[Theorem 2.2.2]{Kr}. In view of Lemma \ref{Radical} there is a minimal left almost split morphism $g:Q \rt Y$ in $\rm{Gprj}\mbox{-} \La$. Hence  $G$ is a direct  summand of $Y.$ We know that $\mmod \La$ contains only finitely many indecomposable projective modules up to isomorphism. As we have seen any indecomposable   non-projective Gorenstein projective module is related to an indecomposable projective module via a minimal left almost split morphism. Because of the uniqueness of the minimal left almost split morphisms, we get the result. We are done.
\end{proof} 

We recall from \cite[Lemma 3.4]{C} that for a semisimple abelian category $\CA $ and an auto-equivalence $\Sigma $ on $\CA$, there is an unique triangulated structure on $\CA$ with $\Sigma$ the translation functor. Indeed, all the triangles are split. We denote  the resulting triangulated category by $(\CA, \ \Sigma)$. We call a triangulated category  semisimple  provided  that it is triangle equivalent to $(\CA, \ \Sigma)$ for some semisimple abelian category.

Recall from \cite{RV} that a Serre functor for a $R$-linear triangulated category ($R$ is a commutative artinian ring) $\CT$ is an auto-equivalence $S:\CT\rt \CT$ 
together with an isomorphism $D\Hom_{\CT}(X, -)\simeq \Hom_{\CT}(-, S(X))$ for each $X \in \CT$,  here  $D$ denotes the ordinary duality.

\begin{proposition}\label{Omega-algebra}
	Let $\La$ be an  $\Omega_{\CG}$-algebra. Then
	\begin{itemize}
		
		\item[$(i)$] $\underline{\rm{Gprj}}\mbox{-} \La$ is a semisimple triangulated category;
		
		\item[$(ii)$] For any non-projective Gorenstein projective indecomposable $G$,  the (relative) Auslander-Reiten translation in $\rm{Gprj} \mbox{-} \La$ is the first syzygy $\Omega_{\La}(G);$
		
		\item[$(iii)$] Serre functor over $\underline{\rm{Gprj}}\mbox{-} \La$ acts on objects by the identity, i.e. $S_{\CG}(G)\simeq G$ for each $G$ in $\underline{\rm{Gprj}}\mbox{-} \La$.	
	\end{itemize}		
\end{proposition}
\begin{proof}
$(i)$	Assume that $\La$ is  an $\Omega_{\CG}$-algebra. Then by Lemma \ref{Lemma1} we have: any short exact sequence with all terms in $\rm{Gprj} \mbox{-} \La$  can be written as a direct sum of the short exact sequences of the form  $0 \rt \Omega_{\La}(A) \rt P \rt A \rt 0$, $0 \rt 0 \rt B  \st{1} \rt B \rt 0$ and $0 \rt C \st{1} \rt C \rt 0 \rt 0$
	for some $A, B $ and $C$ in $\rm{Gprj} \mbox{-} \La.$ Hence by the structure of the  triangles in $\underline{\rm{Gprj}}\mbox{-} \La$,  which is  induced by the short exact sequences in $\rm{Gprj}\mbox{-} \La$, and in view of the shape of the short exact sequences in $\rm{Gprj}\mbox{-}\La$ as we mentioned in the above whenever $\La$ is an $\Omega_{\CG}$-Algebra,  we get  all possible triangles in $\underline{\rm{Gprj}}\mbox{-} \La$ are a direst sums of the following trivial triangles $K \rt 0 \rt  K[1] \st{1} \rt K[1] $, $K \st{1} \rt K \rt 0 \rt K[1]$ and $0 \rt K \st{1} \rt  K \rt 0.$ By the axioms of triangulated categories, every morphism $f:X \rt Y$ in $\underline{\rm{Gprj}}\mbox{-} \La$  is completed to  a triangle, hence, by the above,   we can   write as $f=\bigl(\begin{smallmatrix}
	f_1 & 0 \\
	0 & 0
	\end{smallmatrix}\bigr):X_1\oplus Y_1 \rt X_2 \oplus Y_2$, where $f_1:X_1 \rt X_2$ is an isomorphism in $\underline{\rm{Gprj}}\mbox{-} \La$. Define the natural injection $i:Y_1 \rt X_1 \oplus Y_1$  and the natural projection $X_2 \oplus Y_2 \rt Y_2$ as kernel and cokernel of $f$, respectively. Then this gives a semisimple abelian structure over $\underline{\rm{Gprj}}\mbox{-} \La$.	 Now if we consider the equivalence $\Omega_{\La}:\underline{\rm{Gprj}}\mbox{-} \La \rt \underline{\rm{Gprj}}\mbox{-} \La$, then the resulting triangulated category $(\underline{\rm{Gprj}}\mbox{-} \La, \Omega)$ is triangle equivalent to $\underline{\rm{Gprj}}\mbox{-} \La$, so we get $(i).$
	
	$(ii)$  follows from Lemma \ref{Lemma1} since any almost split sequence does not split. We can deduce  $ (iii)$ in view of $(ii)$  and  using this point that $\tau_{\CG}=S_{\CG}\circ \Omega_{\La} .$ 
\end{proof}	

The above proposition explains for a justification  of the terminology of $\Omega_{\CG}$-algebras. The notation ``$\Omega_{\CG}$" stand for the Auslander-Reiten translation $\tau_{\CG}$ coincides with the first syzygy $\Omega_{\La}$.

\begin{proposition}\label{Symmetry}
	Let $\La$ be  an $\Omega_{\CG}$-algebra. Then $\La^{\rm{op}}$ is also an $\Omega_{\CG}$-algebra. 
\end{proposition}
\begin{proof}
	It is easily proved by using the duality $(-)^*:\rm{Gprj} \mbox{-} \La \rt \rm{Gprj} \mbox{-} \La^{\rm{op}}.$ 
\end{proof}
Let $\La$ be an $\Omega_{\CG}$-algebra. As there are only a finite number of  non-isomorphic indecomposable  Gorenstein projective  modules (up to isomorphism) and the syzygy functor preserves the indecomposable Gorenstein projective modules (\cite[Lemma 2.2]{C1}), so we get the non-projective Gorenstein projective modules are $\Omega_{\La}$-periodic modules, meaning that there exists $n > 0$, depending on a given non-projective Gorenstein projective module $G$,  such that $\Omega^n_{\La}(G)\simeq G$. 

On indecomposable non-projective modules in $\text{Gprj}\mbox{-}\La$, we define an equivalence relation, that is, $G\sim G'$ if and only if there is some  $n \in \mathbb{Z} $ such that $G\simeq \Omega^n_{\La}(G')$. Note that here $\Omega^n_{\La}(G)$ for $n < 0$ are defined by using right (minimal) projective resolution of $G$. In fact, since $\rm{Grpj} \mbox{-}\La$ is a Frobenius category with $\rm{prj}\mbox{-}\La$ as projective-injective objects,  we can construct a right projective resolution by $\rm{prj}\mbox{-}\La.$ 

 Let $\mathcal{C}(\La)$ denote the set of the equivalence classes  with respect to  the relation $\sim$. Moreover,  for an indecomposable non-projective Gorenstein  projective  module $G$ we write $l(G)$ for  the minimum number $n> 0$ such that $\Omega^n_{\La}(G)\simeq G$. One can see the equivalence class $[G]$ has $l(G)$ elements up to isomorphism.

 	Let $G$ be an indecomposable non-projective module in $\text{Gprj}\mbox{-}\La$. We point out if $\delta_1: 0 \rt \Omega_{\La}(G)\rt P\rt G\rt 0$ is an almost split sequence in $\text{Gprj}\mbox{-}\La$, then for every $i\geq 0$, the short exact sequence $0 \rt \Omega^{i+1}_{\La}(G)\rt P^i\rt \Omega^i_{\La}(G)\rt 0$, setting $P^0=P$, is an almost split sequence in $\text{Gprj}\mbox{-}\La$. Indeed, the middle term of the almost split sequence  $0 \rt \tau_{\CG}(\Omega_{\La}(G))\rt B\rt \Omega_{\La}(G)\rt 0$ in $\text{Gprj}\mbox{-}\La$ has no non-projective direct summand, otherwise the middle term of $\delta_1$ has an indecomposable non-projective direct summand, which is a contradiction. Hence $B$ is projective so $\tau_{\CG}(\Omega_{\La}(G))\simeq\Omega^2_{\La}(A)$. Repeating the same argument we can prove for  all  $i \geq 0$. Dually, 	if the short exact sequence $\delta_2: 0 \rt G\rt Q\rt \Sigma_{\La}(G)\rt 0$ is an almost split sequence in $\text{Gprj}\mbox{-}\La$, then for every $i\geq 0$, the short exact sequence $0 \rt \Sigma_{\La}^{i}(G)\rt Q^i\rt \Sigma_{\La}^{i+1}(G)\rt 0$, setting $Q^0=Q$, is an almost split sequence in $\text{Gprj}\mbox{-}\La$. Lastly, If both the short exact sequences $\delta_1$ and $\delta_2$ are almost split sequences in $\text{Gprj}\mbox{-}\La$, then the $\tau_{\CG}$-orbit of $G$ consists entirely of $\Omega^i_{\La}(G)$ and $\Sigma_{\La}^i(G)$ for all $i\geq 0.$ Therefore, in this case the $\tau_{\CG}$-orbit of $G$ is equal to the equivalence class $[G]$ in $\CC(\La)$ (see also Propositions \ref{Prop 3.5} and \ref{Propo 3.7}  for an alternative way to get the above observation).

 \begin{remark}\label{remark 2.7}
 Let $\La=k\CQ/I$ be a monomial algebra, that is, the quotient algebra of the path algebra $k\CQ$	of a finite quiver $\CQ$ modulo the ideal $I$ generated by the paths of length at least 2. Following \cite[Definition 3.7]{CSZ}, we recall that a non-zero path $p$  in $\La$ a perfect path, provided that there exists a sequence
 $p = p_1, p_2,\cdots ,p_n, p_{n+1}= p$  of non-zero paths such that $(p_i, p_{i+1})$ are perfect pairs (see \cite[Definition 3.3]{CSZ}) for all $1\leqslant i \leqslant n$. If the given
 non-zero paths $p_i$ are pairwise distinct, we refer to the sequence $p = p_1, p_2,...,p_n$  as a relation cycle for $p$, which has length $n$. It is proved in \cite[Theorem 4.1]{CSZ}  that there is a bijection between the set of perfect paths
 in $\La$ and the isoclass set of indecomposable non-projective Gorenstein-projective
 $\La$-modules, which is a generalization of the classification  over gentle algebras given  in \cite{Ka}. The bijection is given by  sending a perfect path $p$ to the $A$-module $pA$, where the right ideal $pA$ is generated by all paths $q \notin I$ such that $q=pq'$ for some path $q'$. 
 
 The situation is better when we concentrate on the quadratic monomial algebras. We recall that $\La$ is quadratic monomial provided that the ideal $I$ is generated by paths of length 2. In the rest of this remark assume  $\La$ is a quadratic monomial algebra.  By \cite[lemma 3.4]{CSZ} 
 for a perfect pair $(p, q)$ in $\La$, both $p$ and $q$ are necessarily arrows. In particular, a perfect path is an arrow and its relation cycle consists entirely of arrows. Then by using \cite[Theorem 4.1]{CSZ} we get there is a bijection between perfect arrows in $\La$ and the isoclass set of indecomposable non-projective Gorenstein-projective
 $\La$-modules. 
 
 In \cite[Definition 5.1]{CSZ} the relation quiver $\CR_{\La}$ of $\La$ is defined as follows: the vertices are given by the arrows in $Q$, and there  is an arrow $[\beta\alpha]:\alpha\rt \beta$ if the head of $\beta$ is equal to the tail of $\alpha$  and $\beta\alpha$ lies in $I$. A component $\mathcal{D}$ of $\CR_{\La}$ is called perfect if  it is a basic cycle \cite[definition 5.1]{CSZ}. From \cite[Lemma 5.3]{CSZ}
  we observe that an arrow $\alpha$ is perfect if and only if the corresponding vertex in $\CR_{\La}$ belongs   to a perfect component. As mentioned in the proof of \cite[Theorem 5.7]{CSZ}, for a perfect arrow $\alpha$ lying in the perfect component $\CC$, its relation cycle is of the form $\alpha = \alpha_1, \alpha_2,\cdots, \alpha_{d_i}$ with $\alpha_{d_{i+1}}=\alpha$.  And we have $\Omega_{\La}(\alpha_i\La)=\alpha_{i+1}\La$ and $\Sigma(\alpha_i\La)=\alpha_{i-1}\La.$ According to the  above observation, we see that the set formed by all perfect components of the relation quiver $\CR_{\La}$ are in   one-to-one correspondence with  the elements of $\CC(\La)$, by sending  a given perfect component $\mathcal{D}$ to the equivalence class $[\alpha\La]$ in $\mathcal{C}(\La)$, where  $\alpha$ is a vertex of $\mathcal{D}$.  	
 \end{remark}

Let $\CT$ be an additive category and $F:\CT \rt \CT$ an automorphism. Following \cite{K}, the orbit category $\CT/F$ has the same objects as $\CT$ and its morphisms from $X$ to $Y$ are $\oplus_{n \in \mathbb{Z}}\Hom_{\CT}(X, F^n(Y) ).$ Further, the singularity category $\mathbb{D}_{\text{sg}}(\La)$ of an algebra $\La$ is defined as the Verdier quotient of the bonded derived category  $\mathbb{D}^{\text{b}}(\mmod \La)$ by the subcategory of bounded complexes of projective modules.

\begin{proposition}\label{proposition 2.8}
	Let $\La$ be an $\Omega_{\CG}$-algebra over an algebraic closed filed $k$. If $\La$ is a Gorenstein algebra, then there is an equivalence of triangulated categories $$\mathbb{D}_{\text{sg}}(\La)\simeq \prod_{[G]\in \mathcal{C}(\La)} \frac{\mathbb{D}^{\rm{b}}(\mmod k)}{l(G)} $$ where $\frac{\mathbb{D}^{\rm{b}}(\mmod k)}{l(G)}$ denotes the triangulated orbit category of $\mathcal{D}^{\rm{b}}(\mmod k)$ with respect to the functor $F=[l(G)]$, the $l(G)$-the power of the shift functor.	
\end{proposition}
\begin{proof}
	By the Buchweitz's equivalence over Gorenstein algebras, $\mathcal{D}_{sg}(\La)\simeq \underline{\rm{Gprj}}\mbox{-} \La$, so it is enough to describe $\underline{\rm{Gprj}}\mbox{-} \La$. Due to the lines before Remark \ref{remark 2.7} we see than any indecomposable non-projective in $\rm{Gprj}\mbox{-} \La$ is isomorphic to an module in some equivalence class $[G]$. For any equivalence class $[G] \in \mathcal{C}(\La)$, we have   $\sum^{l(G)}(G')\simeq G'$ in $\underline{\rm{Gprj}}\mbox{-} \La$ for any $G' \in [G].$ Recall that the suspension functor $\sum$ on $\underline{\rm{Gprj}}\mbox{-} \La$ is the quasi-inverse of the syzygy $\Omega_{\La}$. On the other hand,  by the  proof of Proposition \ref{Omega-algebra}, we  deduce that any morphism between two objects in $\underline{\rm{Gprj}}\mbox{-} \La$ are either split monomorphism or split epimorphism. Hence for every indecomposable modules $\underline{X}$ and $\underline{Y}$ in $\underline{\rm{Gprj}} \mbox{-} \La$, $\underline{\Hom}_{\La}(\underline{X}, \underline{Y})=0$ if $\underline{X}\ncong \underline{Y}$, and each non-zero morphism $\underline{f}$ in $\underline{\Hom}_{\La}(\underline{X}, \underline{Y})$ is an isomorphism if $\underline{X} \cong \underline{Y}.$ In the latter case, $\underline{\Hom}_{\La}(\underline{X}, \underline{Y})\simeq \underline{\Hom}_{\La}(\underline{X}, \underline{X})$ (or $\simeq \underline{\Hom}_{\La}(\underline{Y}, \underline{Y})$) as $k$-vector spaces.  But $\underline{\rm{End}}_{\La}(\underline{X})$ is a division algebra containing $k,$ as $k$ is algebraic closed,  this implies that $\underline{\rm{End}}_{\La}(\underline{G})\simeq k$ as $k$-vector space (or more as algebras). Thus,  $\underline{\Hom}_{\La}(\underline{X}, \underline{Y})\simeq k$. Now the rest of the proof goes through as for \cite[Theorem 2.5(b)]{Ka}.
\end{proof}
The above result is inspired by the similar equivalence in \cite{Ka} for gentle algebras and for quadratic monomial algebras in \cite{CSZ}.
\begin{remark}\label{Remark 2.8}
	\begin{itemize}
		\item [$(1)$] The algebras of finite global
dimension are triviality $\Omega_{\CG}$-algebras, since  $\rm{Gprj} \mbox{-} \La= \rm{proj} \mbox{-} \La$. Although, our results  in this paper are not interesting for them. 

 \item[$(2)$] Let $\La = kQ/I$ be a quadratic monomial algebra. In \cite[Theorem 5.7]{CSZ}, it is proved that there is a triangle equivalence
$$\underline{\rm{Gprj}} \mbox{-} \Lambda \simeq \CT_{d_1} \times \CT_{d_2} \times \cdots \times \CT_{d_m}$$
where $\CT_{d_n}=(\mmod k^n, \sigma^n)$ for some natural number $n,$ and auto-equivalence $\sigma_n: \mmod k^n \rt \mmod k^n.$ Thus by the above  equivalence a quadratic monomial algebra is an $\Omega_{\CG}$-algebra, as for each $m$, $\mmod (\mmod k^m)$ is semisimple. Especially, gentle algebras are $\Omega_{\CG}$-algebras.
\item[$(3)$] Let $\La$ be a simple gluing algebra of $A$ and $B$. See the  papers \cite{Lu1} and \cite{Lu2} for two different ways of defining gluing algebras. In these two mentioned papers, for both types of gluing, is proved that $\underline{\rm{Gprj}} \mbox{-} \Lambda\simeq\underline{\rm{Gprj}} \mbox{-} A \coprod \underline{\rm{Gprj}}\mbox{-} B$. Therefore, if $\mmod (\underline{\rm{Gprj}} \mbox{-} A)$ and $\mmod (\underline{\rm{Gprj}} \mbox{-} B)$ are semisimple, then $\mmod( \underline{\rm{Gprj}} \mbox{-} \Lambda)$, by definition,  so is. So simple gluing operation preserve $\Omega_{\CG}$-algebras. For example, cluster-tilted algebras of type $\mathbb{A}_n$ and endomorphism algebras of maximal rigid objects of cluster tube $\mathcal{C}_n$ can be built as simple gluing $\Omega_{\CG}$-algebras , see \cite[ Corollaries 3.21 and 3.22]{Lu2}. Hence this kind of the cluster-tilted algebras are $\Omega_{\CG}$-algebras. 
\item [$(4)$]By Theorem 3.24 in \cite{AHV}, if $\Lambda$ and $\Lambda'$ are derived equivalent, then $\underline{\rm{Gprj}} \mbox{-} \Lambda \simeq \underline{\rm{Gprj}} \mbox{-} \Lambda'.$ Therefore, the property of being semisimple of $\mmod (\underline{\rm{Gprj}} \mbox{-} \La)$ and $\mmod (\underline{\rm{Gprj}} \mbox{-} \La')$  is  preserved under the derived equivalence. Hence $\Omega_{\CG}$-algebras are closed under derived equivalences.
\item [$(5)$]  There are many examples of $\rm{CM}$-finite not-$\rm{\Omega_{\CG}}$-algebras. Let us first give a remark. It was shown in \cite[Theorem 10.7]{AR1}, for an arbitrary Artin algebra $\La$, $\mmod (\underline{\rm{mod}}\mbox{-}\La)$ is a semisimple abelian category if and only if $\La$ is Nakayama algebra  with loewy length at most $2.$ By using this fact, we see that the self-injective Nakayama algebras $k[x]/(x^n)$, $n >2$, are examples of $\rm{CM}$-finite not-$\rm{\Omega_{\CG}}$-algebras.
\end{itemize}
\end{remark}

\section{The monomorphism category over $\Omega_{\mathcal{G}}$-algebras}\label{Section 3}
In this section, all the almost split sequence in $\CS(\text{Gprj}\mbox{-}\La)$ over an $\Omega_{\CG}$-algebra will be computed. As an interesting application we will provide  a nice description of the singularity category of $T_2(\La)$

The following result was first appeared in a primary version of \cite{H} available on arXiv. This result was omitted in the published version of it.

\begin{proposition}
	Assume  the  subcategory of Gorenstein projective modules $\rm{Gprj}\mbox{-}\La$ is contravariantly finite in $\mmod \La$. Then $\CS(\rm{Gprj}\mbox{-}\La)$ is functorially finite in $\rm{H}(\La).$ In particular, $\CS(\rm{Gprj}\mbox{-}\La)$ has almost split sequences.
\end{proposition}
\begin{proof}
	Since $\CS(\rm{Gprj}\mbox{-}\La)$ is indeed the subcategory of Gorenstein projective modules in $\rm{H}(\La)$ and so a resolving subcategory, then by \cite[Corollary 0.3 ]{KS} it is enough to show that $\CS(\rm{Gprj}\mbox{-}\La)$ is contravariantly finite in $\rm{H}(\La).$ We know from \cite[Theorem 2.5]{RS} that $\CS(\La)$ is contravariantly finite in $\rm{H}(\La)$, hence by making use of this fact,   it is enough to show 
	that any object in $\CS(\La)$ has a right $\CS(\rm{Gprj}\mbox{-}\La)$-approximation. Take an arbitrary object $\left(\begin{smallmatrix}M_1\\ M_2\end{smallmatrix}\right)_f$  in $\CS(\La)$. Hence $f$ is a monomorphism. Let $G_1 \st{g}\rt \text{Cok}(f)\rt 0$ be a minimal right $\rm{Gprj}\mbox{-}\La$-approximation of $\text{Cok}(f)$ in $\mmod \La,$ which it exists by our assumption. Further, set $K_1:=\Ker(g)$, since $g$ is minimal then by Wakamutsu's Lemma,   $K_1 \in \rm{Gprj}\mbox{-}\La^{\perp}$, i.e., $\Ext^1(G, K_1)=0$ for any $G$ in $\rm{Gprj}\mbox{-}\La$. Consider the following pull-back diagram
	$$\xymatrix{&   & 0\ar[d] & 0\ar[d] & &\\
		&   & M_1 \ar[d]^h
		\ar@{=}[r] & M_1 \ar[d]^{f} \ar[r] & 0\\
		0 \ar[r] & K_1\ar@{=}[d] \ar[r] & U
		\ar[d]^l \ar[r]^{d} & M_2 \ar[d]^{\pi} \ar[r] & 0\\
		0 \ar[r] & K_1 \ar[d] \ar[r]^{\mu_1} & G_1
		\ar[d] \ar[r]^{g} & \text{Cok}(f) \ar[d] \ar[r] & 0\\
		& 0  & 0  & 0 & }	$$
	Let $G_2 \st{t}\rt U \rt 0$ be a minimal right $\rm{Gprj}\mbox{-}\La$-approximation. Again by applying Wakamutsu's Lemma, $K_2:= \text{Ker}(t)$ belongs to $\rm{Gprj}\mbox{-}\La^{\perp}$. Now consider the following pull-back diagram	
	$$\xymatrix{&   & 0 \ar[d] & 0 \ar[d]& &\\
		0 \ar[r] & K_2 \ar@{=}[d] \ar[r] & K_3 \ar[d]
		\ar[r] &K_1\ar[d] \ar[r] & 0\\
		0 \ar[r] & K_2  \ar[r] & G_2
		\ar[d]^{dt} \ar[r]^{t} & U\ar[d]^d \ar[r] & 0\\
		&  & M_2
		\ar[d] \ar@{=}[r] & M_2 \ar[d] & \\
		&  & 0  & 0 & }
	$$
	Since both $K_1, K_2$ are in $\rm{Gprj}\mbox{-}\La^{\perp}$, so the first row in above digram follows $K_3$ so is.  Finally, applying the Snake lemma for the following commutative diagram
	$$\xymatrix{
		0 \ar[r] &0 \ar[d] \ar[r] & G_2 \ar[d]^t
		\ar@{=}[r] &G_2 \ar[d]^{lt}\ar[r] & 0&\\
		0 \ar[r] & M_1 \ar[r]^{h} & U
		\ar[r]^l & G_1\ar[r] &0.& \\ 	} 	 $$
	yields the following short exact sequence 
	$$0 \rt K_2 \rt G_3\rt M_1\rt 0, $$
	where  $G_3:= \text{Ker}(lt)$, 	that is a Gorenstein projective module since it is a kernel of an epimorphism in $\rm{Gprj}\mbox{-}\La$.
	Putting together the maps   obtained  in the above,  the following commutative diagram can be constructed
	$$\xymatrix{& 0 \ar[d] & 0 \ar[d] & 0 \ar[d]&(*) &\\
		0 \ar[r] & K_2 \ar[d]^v \ar[r] & G_3 \ar[d]^w
		\ar[r]^p & M_1 \ar[d]^{f} \ar[r] & 0\\
		0 \ar[r] & K_3\ar[d] \ar[r] & G_2
		\ar[d]^{lt} \ar[r]^{dt} & M_2 \ar[d]^{\pi} \ar[r] & 0\\
		0 \ar[r] & K_1 \ar[d] \ar[r] & G_1
		\ar[d] \ar[r]^{g} & \rm{Cok}(f) \ar[d] \ar[r] & 0\\
		& 0  & 0  & 0 & }$$
	
	From the above digram, we  obtain the following short exact sequence in $\rm{H}(\La)$
	
	$$\xymatrix@1{\la: \ \  0\ar[r] & {\left(\begin{smallmatrix}K_2\\ K_3\end{smallmatrix}\right)_v}
		\ar[rr]
		& & {\left(\begin{smallmatrix}G_3\\ G_2 \end{smallmatrix}\right)_w}\ar[rr]^{\left(\begin{smallmatrix}p\\ dt \end{smallmatrix}\right)}& &
		{\left(\begin{smallmatrix}M_1\\ M_2\end{smallmatrix}\right)_f}\ar[r]& 0 } \ \    $$	
	such that the middle term lies in $\CS(\rm{Gprj}\mbox{-}\La)$ and $K_2, K_3$ are in $\rm{Gprj}\mbox{-}\La^{\perp}$. Consider an arbitrary object $\left(\begin{smallmatrix}G\\ G'\end{smallmatrix}\right)_s$  in  $\CS(\rm{Gprj}\mbox{-}\La)$. In view of the following short exact sequence
$$\xymatrix@1{ 0\ar[r] & {\left(\begin{smallmatrix}G\\ \text{Im}(s)\end{smallmatrix}\right)_j}
	\ar[rr]^{\left(\begin{smallmatrix}1\\ i \end{smallmatrix}\right)}
	& & {\left(\begin{smallmatrix}G\\ G' \end{smallmatrix}\right)_s}\ar[rr]^{\left(\begin{smallmatrix}0\\ \pi' \end{smallmatrix}\right)}& &
	{\left(\begin{smallmatrix}0\\ \text{Cok}(s)\end{smallmatrix}\right)_0}\ar[r]& 0 } \ \    $$	
	\noindent	
	where $ij=s$ is an epi-mono factorization of $s$ and $\pi'$ the canonical epimorphism, consequently $j$ is an isomorphism. The sequence   implies that $\Ext^1_{\rm{H}(\La)}(\left(\begin{smallmatrix}G\\ G' \end{smallmatrix}\right)_s, \left(\begin{smallmatrix}K_2\\ K_3 \end{smallmatrix}\right)_v)=0$ for any object $\left(\begin{smallmatrix}G\\ G' \end{smallmatrix}\right)_s$ in $\CS(\rm{Gprj}\mbox{-}\La)$. Indeed, by the known adjoint pairs between $\rm{H}(\La)$ and $\mmod \La$ along  with the vanishing of $\Ext$ for $K_2, K_3$ , we have $\Ext^1_{\rm{H}(\La)}(\left(\begin{smallmatrix}G\\ \text{Im}(s) \end{smallmatrix}\right)_j, \left(\begin{smallmatrix}K_2\\ K_3 \end{smallmatrix}\right)_v )=\Ext^1_{\La}(G, K_2)=0$ and $\Ext^1_{\rm{H}(\La)}(\left(\begin{smallmatrix}0\\ \text{Cok}(s) \end{smallmatrix}\right)_0,\left(\begin{smallmatrix}K_2\\ K_3 \end{smallmatrix}\right)_v )=\Ext^1_{\La}(\text{Cok}(s), K_3)=0$, and consequently, the desired vanishing of $\Ext$ in $\rm{H}(\La)$. This follows the epimorphism included in the sequence $(\lambda)$ is a right $\CS(\rm{Gprj}\mbox{-}\La)$-approximation of $\left(\begin{smallmatrix}M_1\\ M_2\end{smallmatrix}\right)_f$ in $\rm{H}(\La)$. Now the proof is complete.
\end{proof}
We need the following lemmas.
\begin{lemma}\label{Lemma 4.2}
	\begin{itemize}
		\item [$(1)$]Let $P$ be an indecomposable projective module in $\mmod \La$ with radical $\text{rad}P$.The objects $\left(\begin{smallmatrix}
		0 \\ P
		\end{smallmatrix}\right)_0$ and $\left(\begin{smallmatrix}
		P \\ P
		\end{smallmatrix}\right)_1$  are indecomposable projective  objects in $\CS(\La)$ (and more in $\text{H}(\La)$). Each indecomposable (relatively) projective object arises in this way. Further, the following inclusions play as sink maps in $\CS(\La)$ (and more in $\text{H}(\La)$)
		$$\left(\begin{smallmatrix}
		0 \\ i
		\end{smallmatrix}\right):\left(\begin{smallmatrix}
		0 \\ \text{rad} P
		\end{smallmatrix}\right)_0\rt \left(\begin{smallmatrix}
		0\\P
		\end{smallmatrix} \right)_0 \ \  \text{and}  \ \ \  \left(\begin{smallmatrix}
		i \\ 1
		\end{smallmatrix}\right):\left(\begin{smallmatrix}
		\text{rad} P\\ P
		\end{smallmatrix}\right)_i\rt \left(\begin{smallmatrix}
		P\\P
		\end{smallmatrix} \right)_1,
		$$
		where $i:\text{rad}P\hookrightarrow P $ the canonical inclusion, respectively.
		\item [$(2)$] Let $I$ be an indecomposable injective module in $\mmod \La$ with socle  $\text{soc}I$. Then the objects $\left(\begin{smallmatrix}
		I \\ I
		\end{smallmatrix}\right)_1$ and $\left(\begin{smallmatrix}
		0 \\ I
		\end{smallmatrix}\right)_0$  are   indecomposable injective objects in $\CS(\La)$. Each indecomposable (relatively) injective object in  $\CS(\La)$ arises in this way. Further,  the following inclusions play as source maps in $\CS(\La)$ 
		$$\left(\begin{smallmatrix}
		p \\ p
		\end{smallmatrix}\right):\left(\begin{smallmatrix}
		I \\ I
		\end{smallmatrix}\right)_1\rt \left(\begin{smallmatrix}
		I/\text{soc} I\\ I/\text{soc} I
		\end{smallmatrix} \right)_1 \ \  \text{and}  \ \ \ \left(\begin{smallmatrix}
		0 \\ 1
		\end{smallmatrix}\right): \left(\begin{smallmatrix}
		0 \\ I
		\end{smallmatrix}\right)\rt \left(\begin{smallmatrix}
		\text{soc} I\\I
		\end{smallmatrix} \right)_j,
		$$
		where $p:I\rt I/\text{soc} I$ the canonical quotient and $j:\text{soc}I\hookrightarrow I$ the inclusion, 	respectively. 
	\end{itemize}
\end{lemma}
\begin{proof}
See \cite[Proposition 1.4]{RS}.
\end{proof}

\begin{lemma}\label{Lemma 4.3}
	\begin{itemize}
		
		\item [$(1)$]Let $P$ be an indecomposable projective module in $\mmod \La$ with radical $\text{rad}P$.The objects $\left(\begin{smallmatrix}
		0 \\ P
		\end{smallmatrix}\right)$ and $\left(\begin{smallmatrix}
		P\\ P
		\end{smallmatrix}\right)_1$   are indecomposable projective-injective objects. Each indecomposable injective object arises in this way. 
		\item [$(2)$] The following composition map is sink in $\CS(\text{Gprj}\mbox{-}\La)$
			$$\left(\begin{smallmatrix}
		0\\ i
		\end{smallmatrix}\right)\circ \left(\begin{smallmatrix}
		0\\ f
		\end{smallmatrix}\right):\left(\begin{smallmatrix}
		0\\ G
		\end{smallmatrix}\right)_0\rt \left(\begin{smallmatrix}
		0\\ \text{rad}P
		\end{smallmatrix}\right)_0\rt \left(\begin{smallmatrix}
		0\\P
		\end{smallmatrix} \right)_0 $$
		where
		 $f:G\rt \text{rad}P$ a  minimal right $\rm{Gprj}\mbox{-}\La$-approximation and $i:\text{rad}P\rt P$ the canonical inclusion.		
			\end{itemize}	
\end{lemma}
\begin{proof}
$(1)$ See \cite[Proposition 1]{HM}. Set $K:=\text{Ker}(f)$. By Wakamutsu's Lemma we have $K \in \text{Gprj}\mbox{-}\La^{\perp}$. For any $\left(\begin{smallmatrix}
G\\ G'
\end{smallmatrix}\right)_e$ in $\CS(\text{Gprj}\mbox{-}\La)$, by applying $\Ext^1_{\text{H}(\La)}(-, \left(\begin{smallmatrix}
0\\ K
\end{smallmatrix}\right)_0 )$ on the following short exact sequence

$$0 \rt \left(\begin{smallmatrix}
G\\ \text{Im}(e)
\end{smallmatrix}\right)_g\rt \left(\begin{smallmatrix}
G\\ G'
\end{smallmatrix}\right)_e\rt \left(\begin{smallmatrix}
0\\ \text{Cok}(e)
\end{smallmatrix}\right)_0\rt 0$$
where, by construction, $g$ is an isomorphism, we obtain $\Ext^1_{\text{H}(\La)}(\left(\begin{smallmatrix}
G\\ G'
\end{smallmatrix}\right)_e, \left(\begin{smallmatrix}
0\\ K
\end{smallmatrix}\right)_0 )=0$. Hence $\left(\begin{smallmatrix}
0\\ K
\end{smallmatrix}\right)_0 \in \CS(\text{Gprj}\mbox{-}\La)^{\perp}$. This implies that $\left(\begin{smallmatrix}
0\\ f
\end{smallmatrix}\right) :\left(\begin{smallmatrix}
0\\ G
\end{smallmatrix}\right) \rt \left(\begin{smallmatrix}
0\\ \text{rad}P
\end{smallmatrix}\right)$ is a right $\CS(\text{Gprj}\mbox{-}\La)$-approximation, an moreover minimal in $\CS(\text{Gprj}\mbox{-}\La)$ by using the (right) minimality of $f$ in $\mmod \La$. Now by using the fact we have already proved for the morphism $\left(\begin{smallmatrix}
0\\ f
\end{smallmatrix}\right)$ being a minimal right  $\CS(\text{Gprj}\mbox{-}\La)$-approximation, we will show that $\left(\begin{smallmatrix}
0\\ if
\end{smallmatrix}\right)$ is right minimal almost split (sink). If $\left(\begin{smallmatrix}
0\\ if
\end{smallmatrix}\right)\circ\left(\begin{smallmatrix}
0\\ h
\end{smallmatrix}\right)=\left(\begin{smallmatrix}
0\\ if
\end{smallmatrix}\right)$, then $hf=f$. Because of being right minimal of $f$, we get $h$ is an isomorphism, consequently  so is  $\left(\begin{smallmatrix}
0\\ h
\end{smallmatrix}\right)$. Thus $\left(\begin{smallmatrix}
0\\ if 
\end{smallmatrix}\right)$ is  right minimal. Finally, we will show that it is right almost split.
 Assume  $\left(\begin{smallmatrix}
0\\ g
\end{smallmatrix}\right):\left(\begin{smallmatrix}
X\\ Y
\end{smallmatrix}\right)_v\rt \left(\begin{smallmatrix}
0\\ P
\end{smallmatrix}\right)_0$ is not a retraction. If $\left(\begin{smallmatrix}
0\\ g
\end{smallmatrix}\right)$ is not a retraction, then as the inclusion $\left(\begin{smallmatrix}
0\\ i
\end{smallmatrix}\right)$ is sink, there exists $\left(\begin{smallmatrix}
0\\ l
\end{smallmatrix}\right)$ such that $\left(\begin{smallmatrix}
0\\ g
\end{smallmatrix}\right)=\left(\begin{smallmatrix}
0\\ i
\end{smallmatrix}\right)\circ\left(\begin{smallmatrix}
0\\ l
\end{smallmatrix}\right)$. As we proved $\left(\begin{smallmatrix}
0\\ f
\end{smallmatrix}\right)$ is a right $\CS(\text{Gprj}\mbox{-}\La)$-approximation, so there is $\left(\begin{smallmatrix}
0\\ d
\end{smallmatrix}\right):\left(\begin{smallmatrix}
X\\ Y
\end{smallmatrix}\right)_v\rt \left(\begin{smallmatrix}
0\\ G
\end{smallmatrix}\right)_0$ such that  $\left(\begin{smallmatrix}
0\\ l
\end{smallmatrix}\right)=\left(\begin{smallmatrix}
0\\ f
\end{smallmatrix}\right)\circ\left(\begin{smallmatrix}
0\\ d
\end{smallmatrix}\right)$. Hence $\left(\begin{smallmatrix}
0\\ g
\end{smallmatrix}\right)$ factors through $\left(\begin{smallmatrix}
0\\ if
\end{smallmatrix}\right)$ via $\left(\begin{smallmatrix}
0\\ d
\end{smallmatrix}\right)$, as desired. We are done.
\end{proof}
In the rest, we assume $\CM$ is either $\mmod \La$ or $\text{Gprj}\mbox{-}\La.$ According to which category we take then the relevant concept  will be defined. For instance, when we say $A$ is an injective object in $\CM$:  if $\CM=\mmod \La$, then $A \in \text{inj}\mbox{-}\La$, and if $\CM=\text{Gprj}\mbox{-}\La$, then $A \in \text{prj}\mbox{-}\La.$

\begin{lemma}\label{AlmostSplittrivialmonomorphisms}
	Assume  $\delta: 0 \rt  A\st{f} \rt B\st{g} \rt C \rt 0$ is  an almost split sequence in $\CM.$ Then
	\begin{itemize}
		\item[$(1)$] The almost split sequence in $\CS(\CM)$ ending at $\left(\begin{smallmatrix}0\\ C\end{smallmatrix}\right)_0$ has of  the form 
		$$\xymatrix@1{  0\ar[r] & {\left(\begin{smallmatrix} A\\ A\end{smallmatrix}\right)}_{1}
			\ar[rr]^-{\left(\begin{smallmatrix} 1 \\ f\end{smallmatrix}\right)}
			& & {\left(\begin{smallmatrix}A\\ B\end{smallmatrix}\right)}_{f}\ar[rr]^-{\left(\begin{smallmatrix} 0 \\ g\end{smallmatrix}\right)}& &
			{\left(\begin{smallmatrix}0\\ C\end{smallmatrix}\right)}_{0}\ar[r]& 0. } \ \    $$			
		
		\item [$(2)$] Let $e:A \rt I$ be  an injective envelop of $A$. Then the almost split sequence in $\CS(\CM)$ ending at $\left(\begin{smallmatrix}C\\ C\end{smallmatrix}\right)_1$ has of the form
		$$\xymatrix@1{  0\ar[r] & {\left(\begin{smallmatrix} A\\ I\end{smallmatrix}\right)}_{e}
			\ar[rr]^-{\left(\begin{smallmatrix} f \\ \left(\begin{smallmatrix} 1 \\ 0\end{smallmatrix}\right)\end{smallmatrix}\right)}
			& & {\left(\begin{smallmatrix}B\\ I\oplus C\end{smallmatrix}\right)}_{h}\ar[rr]^-{\left(\begin{smallmatrix} g \\ \left(\begin{smallmatrix} 0 & 1\end{smallmatrix}\right)\end{smallmatrix}\right)}& &
			{\left(\begin{smallmatrix}C\\ C\end{smallmatrix}\right)}_{1}\ar[r]& 0. } \ \    $$			 	
		where $h$ is the map $\left(\begin{smallmatrix} e' \\ g\end{smallmatrix}\right)$  with $e':B \rt I$ is an extension of $e.$ Indeed, it is induced from the following push-out diagram
		$$\xymatrix{		 A \ar[d]^e \ar[r]^{f} & B \ar[d]^h
			\ar[r]^{g} & C \ar@{=}[d]  \\
			I \ar[r]^{\left(\begin{smallmatrix} 1 \\ 0\end{smallmatrix}\right)} &I\oplus C\ar[r]^{\left(\begin{smallmatrix} 0 & 1\end{smallmatrix}\right)}
			 & C    }	$$

		\item [$(3)$] Let $p:P\rt C$ be the projective cover of $C$. Then the almost split sequence in $\CS(\CM)$ starting  at $\left(\begin{smallmatrix}0\\ A\end{smallmatrix}\right)_0$ has of the form 
		$$\xymatrix@1{  0\ar[r] & {\left(\begin{smallmatrix} 0\\ A\end{smallmatrix}\right)}_{0}
			\ar[rr]^-{\left(\begin{smallmatrix} 0 \\ \left(\begin{smallmatrix} 1 \\ 0\end{smallmatrix}\right)\end{smallmatrix}\right)}
			& & {\left(\begin{smallmatrix}\Omega_{\La}(C)\\ A\oplus P\end{smallmatrix}\right)}_{h}\ar[rr]^-{\left(\begin{smallmatrix} 1 \\ \left(\begin{smallmatrix} 0 & 1\end{smallmatrix}\right)\end{smallmatrix}\right)}& &
			{\left(\begin{smallmatrix}\Omega_{\La}(C)\\ P\end{smallmatrix}\right)}_{i}\ar[r]& 0. } \ \    $$		
		where $h$ is  the kernel of morphism $(f~~p'):A\oplus P\rt B$, here $p'$ is a lifting of $p$ to $g$. Indeed, it is induced from the following pull-back diagram
		
		$$\xymatrix{		  & \Omega_{\La}(C) \ar[d]^h
			\ar@{=}[r] & \Omega_{\La}(C) \ar[d]^i  \\
			A\ar@{=}[d] \ar[r]^{\left(\begin{smallmatrix} 1 \\ 0\end{smallmatrix}\right)} & A\oplus P
			\ar[d]^{\left(\begin{smallmatrix} f & p'\end{smallmatrix}\right)} \ar[r]^{\left(\begin{smallmatrix} 0 & 1\end{smallmatrix}\right)} & P \ar[d]^p  \\
			A  \ar[r]^{f} & B
			\ar[r]^{g} & C }	$$
			
	\end{itemize}
\end{lemma}
\begin{proof}
	We refer to \cite[Lemma 6.3]{H}. 
\end{proof}

\begin{proposition}\label{Prop 3.5}
	Assume  that $N$ is an indecomposable  non-projective module in $\CM$. Suppose that  there exists the almost split sequence $0 \rt M\st{i_0}\rt P\st{p_0}\rt N\rt 0$ in $\CM$ with $P$ projective. If $M$ is non-projective and  let  $0\rt V\st{i_1}\rt Q\st{p_1}\rt M\rt 0$  be the almost split sequence ending at $M$ in $\CM$. Then, the following assertions hold.
	
	\begin{itemize}
		\item [$(1)$] $Q$ is injective in $\CM$;
		
			\item [$(2)$] there are the following almost split sequences in $\CS(\CM)$
		$$\xymatrix@1{  0\ar[r] & {\left(\begin{smallmatrix} 0\\ M\end{smallmatrix}\right)_0}
			\ar[rr]
			& & {\left(\begin{smallmatrix}M\\ M\end{smallmatrix}\right)_1}\oplus {\left(\begin{smallmatrix}0\\ P\end{smallmatrix}\right)_0} \ar[rr]& &
			{\left(\begin{smallmatrix}M\\ P\end{smallmatrix}\right)_{i_0}}\ar[r]& 0. } \ \    $$	
		$$\xymatrix@1{  0\ar[r] & {\left(\begin{smallmatrix} V\\ Q\end{smallmatrix}\right)_{i_1}}
			\ar[rr]
			& & {\left(\begin{smallmatrix}Q\\ Q\end{smallmatrix}\right)_1}\oplus {\left(\begin{smallmatrix}0\\ M\end{smallmatrix}\right)_0}\ar[rr]& &
			{\left(\begin{smallmatrix}M\\ M\end{smallmatrix}\right)_1}\ar[r]& 0; } \ \    $$	
		\item [$(3)$]in particular, there is the following full sub-quiver in $\Gamma_{\CM}$
		
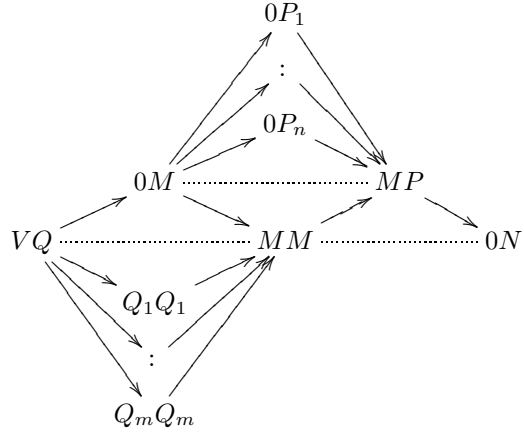
	\[
	\xymatrix  @R=0.3cm  @C=0.6cm {&&&&& 0P_1\ar[rddd]&&&\\ &&&&& \colon\ar[rdd] &&&\\
	&&&&&0P_n\ar[dr]& &&	\\   &&&	&0M\ar[dr]\ar[ruuu]\ar[ruu]\ar[ru]\ar@{.}[rr]&&MP\ar[dr]&&&\\  &&&
		VQ\ar[dr]\ar[ddr]\ar[dddr]\ar@{.}[rr]\ar[ru]&&MM\ar[ru]\ar@{.}[rr]&&0N&&\\
			&&&	&Q_1Q_1\ar[ru]&&&&&\\&&&&\colon\ar[ruu]&&&&\\&&&&Q_mQ_m\ar[ruuu]&&&&}
	\]
 where a decomposition  $P=\oplus_1^nP_i$, resp. $Q=\oplus^m_1Q_i$ of $P$, resp. $Q$,  into indecomposable modules. 
 	
		\end{itemize}

\end{proposition}
\begin{proof}
Assume that $N$ satisfies the condition of  the statement. According to Lemma \ref{AlmostSplittrivialmonomorphisms}(1) there is the  almost split sequence 
$$\xymatrix@1{\la: \ \  0\ar[r] & {\left(\begin{smallmatrix} M\\ M\end{smallmatrix}\right)_1}
	\ar[rr]
	& & {\left(\begin{smallmatrix}M\\ P \end{smallmatrix}\right)_{i_0}}\ar[rr]& &
	{\left(\begin{smallmatrix}0\\ N\end{smallmatrix}\right)_0}\ar[r]& 0. } \ \    $$			
in $\CS(\CM)$. Hence $\left(\begin{smallmatrix}M\\ M \end{smallmatrix}\right)_{1}$  is a direct summand of the middle term of the almost split sequence $\eta$ in $\CS(\CM)$ ending at $\left(\begin{smallmatrix}M\\ P \end{smallmatrix}\right)_{i_0}$.  Again using Lemma \ref{AlmostSplittrivialmonomorphisms}(3) the middle term of $\eta$ is is obtained by the following  pull-back diagram, that is $\left(\begin{smallmatrix}M\\ M\oplus P \end{smallmatrix}\right)_{h},$
$$\xymatrix{		  & M \ar[d]^h
	\ar@{=}[r] & M \ar[d]^{i_0}  \\
	M\ar@{=}[d] \ar[r] & M\oplus P
	\ar[d] \ar[r] & P \ar[d]^{p_0}  \\
	M\ar[r]^{i_0} & P
	\ar[r]^{p_0} & N   }	$$
Since $\left(\begin{smallmatrix}M\\ M \end{smallmatrix}\right)_{1}$ is a direct summand of $\left(\begin{smallmatrix}M\\ M\oplus P \end{smallmatrix}\right)_{h}$ and $M$ being non-projective but $P$ a projective, we easily get the isomorphisms $\left(\begin{smallmatrix}M\\ M\oplus P \end{smallmatrix}\right)_{h}\simeq \left(\begin{smallmatrix}M\\ M \end{smallmatrix}\right)_{1}\oplus \left(\begin{smallmatrix}0\\  P \end{smallmatrix}\right)_{0}=\left(\begin{smallmatrix}M\\ M \end{smallmatrix}\right)_{1}\oplus (\oplus^n_i\left(\begin{smallmatrix}0\\ P_i \end{smallmatrix}\right)_{0})$, so the sequence $\eta$ gets the desired form.

 Similarly, from the almost split sequence  $\eta$ we obtain $\left(\begin{smallmatrix}0\\ M \end{smallmatrix}\right)_{0}$ is a direct summand of the almost split sequence $\eta'$ ending at $\left(\begin{smallmatrix}M\\ M \end{smallmatrix}\right)_{1}$  in $\CS(\CM)$. We infer from Lemma \ref{AlmostSplittrivialmonomorphisms}(2): 
$$\xymatrix@1{\eta': \ \  0\ar[r] & {\left(\begin{smallmatrix} V\\ I\end{smallmatrix}\right)}_e
	\ar[rr]
	& & {\left(\begin{smallmatrix}Q\\ I\oplus M\end{smallmatrix}\right)_d}\ar[rr]& &
	{\left(\begin{smallmatrix}M\\ M\end{smallmatrix}\right)_1}\ar[r]& 0. } \ \    $$
where $e$ is an injective envelop for when $\CM=\mmod \La$ and a minimal left $\rm{prj}\mbox{-}\La$-approximation of $V$ for the case $\CM=\text{Gprj}\mbox{-}\La$. Because of this fact that $\left(\begin{smallmatrix}0\\ M \end{smallmatrix}\right)_{0}$  is a direct summand of the middle term of $\eta'$ and $M$ being non-injective, we easily get the isomorphism $\left(\begin{smallmatrix}Q\\ I\oplus M \end{smallmatrix}\right)_{d}\simeq\left(\begin{smallmatrix}0\\ M \end{smallmatrix}\right)_{0}\oplus\left(\begin{smallmatrix}Q\\ I \end{smallmatrix}\right)_{e'}$. As the middle term $\left(\begin{smallmatrix}M\\ P \end{smallmatrix}\right)_{i_0}$ of the sequence $\la$ is indecomposable, the object  $\left(\begin{smallmatrix}Q\\ I \end{smallmatrix}\right)_{e'}$ must be injective. Thanks to the classification of the injective objects in $\CS(\CM)$ (see  Lemma \ref{Lemma 4.2} for when $\CM=\mmod \La$ and Lemma \ref{Lemma 4.3} for the case $\CM=\text{Gprj}\mbox{-}\La$),  $e'$ is a split monomorphism and $Q$ an injective module, so the statement $(1)$ follows. To complete the proof we have to show that $e'$ has no direct summand of the form $\left(\begin{smallmatrix}0\\ J\end{smallmatrix}\right)_0$  for some injective module $J$. If this case holds, then the map in $\left(\begin{smallmatrix}M\\ M\end{smallmatrix}\right)_1$  would not be the identity map, a contradiction. Indeed, for when $\CM=\mmod \La$, if we assume that there is an indecomposable non-zero direct summand $\left(\begin{smallmatrix}0\\ J\end{smallmatrix}\right)_0$ for some indecomposable injective module $J$,  by Lemma \ref{Lemma 4.2} we have the source map $\left(\begin{smallmatrix}0\\J\end{smallmatrix}\right)_0\rt \left(\begin{smallmatrix}\text{soc}J\\ J\end{smallmatrix}\right)_i$, where $i:\text{soc}J\rt J$ the canonical inclusion,   this means that $\left(\begin{smallmatrix}M\\M\end{smallmatrix}\right)_1\simeq \left(\begin{smallmatrix}\text{soc}J\\ J\end{smallmatrix}\right)_i$, and so $\text{soc}J\simeq J$,  we reach a contradiction. If $\CM=\rm{Gprj}\mbox{-}\La$, then by Lemma \ref{Lemma 4.3} the domain of the sink map ending at $\left(\begin{smallmatrix}
0\\ J
\end{smallmatrix}\right)_0$ is $\left(\begin{smallmatrix}
0\\ G
\end{smallmatrix}\right)_0$, where $G$ is a minimal right $\rm{Gprj}\mbox{-}\La$-approximation of $\text{rad}J$. Hence $V$ would be zero, which is a contradiction. We have proved $(1)$ and $(2)$ so far. The claim $(3)$ follows from $(1)$ and $(2)$ along with Lemma \ref{AlmostSplittrivialmonomorphisms}. So we are done.
\end{proof}
\begin{remark}
\begin{itemize}
	\item [$(1)$] 
For the case $\CM=\mmod \La,$ there is a nice description of the almost split sequence $0 \rt A\rt B \rt C \rt 0$ with $B$ either a projective module or an injective module via whose ending terms as follows (see \cite[\S V Theorem 3.3]{AuslanreitenSmalo}): $(1)$ The module $B$ is injective if and only if  $C$ is a non-projective  simple module which is not a composition factor of $\text{rad}P/\text{soc}P$ for any projective module $P$. $(2)$ The module $B$ is projective if and only if  $C$ is a non-injective  simple module which is not a composition factor of $\text{rad}I/\text{soc}I$ for any injective module $I$.
\item [$(2)$] For $\CM=\text{Gprj}\mbox{-}\La$ we do not know general information when the middle term of an almost split sequence in $\text{Gprj}\mbox{-}\La$, as in the above proposition, is projective. But for the special case, when $\La$ is a monomial algebra, according to  \cite[Lemma 3.1]{CSZ}  the projective module in the middle term must be indecomposable.
\end{itemize}
\end{remark}

The following proposition  is proved by the dual arguments to ones used in the above result.

\begin{proposition}\label{Propo 3.7}
 	Assume  $M$ is an indecomposable  non-injective module in $\CM$. Suppose that  there exists the almost split sequence $0 \rt M\st{j_0}\rt I\st{q_0}\rt N\rt 0$ in $\CM$ with $I$ injective in $\CM$. If $N$ is non-injective and  let  $0\rt N\st{j_1}\rt I'\st{q_1}\rt W\rt 0$  be the almost split sequence starting at $N$ in $\CM$. Then, the following assertions hold.
	\begin{itemize}
		\item [$(1)$] $I'$ is projective in $\CM$;
		
		\item [$(2)$] there are the following almost split sequences in $\CS(\CM)$
		$$\xymatrix@1{  0\ar[r] & {\left(\begin{smallmatrix} M\\ I\end{smallmatrix}\right)_{j_0}}
			\ar[rr]
			& & {\left(\begin{smallmatrix}I\\ I\end{smallmatrix}\right)_1}\oplus {\left(\begin{smallmatrix}0\\ N\end{smallmatrix}\right)_0} \ar[rr]& &
			{\left(\begin{smallmatrix}N\\ N\end{smallmatrix}\right)_1}\ar[r]& 0. } \ \    $$	
		$$\xymatrix@1{  0\ar[r] & {\left(\begin{smallmatrix} 0\\ N\end{smallmatrix}\right)_0}
			\ar[rr]
			& & {\left(\begin{smallmatrix}N\\ N\end{smallmatrix}\right)_1}\oplus {\left(\begin{smallmatrix}0\\ I'\end{smallmatrix}\right)_0}\ar[rr]& &
			{\left(\begin{smallmatrix}N\\ I'\end{smallmatrix}\right)_{j_1}}\ar[r]& 0; } \ \    $$	
		\item [$(3)$]in particular, there is the following full sub-quiver in $\Gamma_{\CS(\CM)}$
		
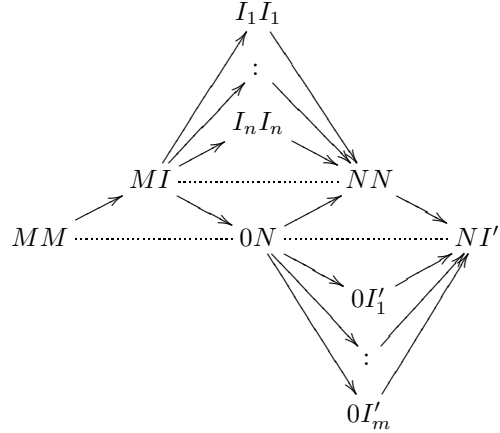
		\[
		\xymatrix  @R=0.3cm  @C=0.6cm { 
	&&&&& I_1I_1\ar[rddd]&&&\\	&&&&& \colon\ar[rdd]&&&\\	&&&&&I_nI_n\ar[dr]& &&	\\   &&&	&MI\ar[dr]\ar[ru]\ar[ruu]\ar[ruuu]\ar@{.}[rr]&&NN\ar[dr]&&&\\  &&&
			MM\ar@{.}[rr]\ar[ru]&&0N\ar[ru]\ar@{.}[rr]\ar[dr]\ar[rdd]\ar[rddd]&&NI'&&\\
			&&&	&&&0I'_1\ar[ur]&&&\\&&&&& &\colon\ar[ruu] &&\\&&&& &&0I'_m \ar[ruuu]&&}
		\]
where a decomposition  $I=\oplus_1^nI_i$, resp. $I'=\oplus^m_1I'_i$ of $I$, resp. $I'$,  into indecomposable modules.				
	\end{itemize}
\end{proposition}
We establish the following construction by specializing the above last two propositions
for when $\CM=\text{Gprj}\mbox{-}\La.$
\begin{construction}\label{Construction 4.8}
Starting from an $N \in \text{Gprj}\mbox{-}\La$ being  indecomposable non-projective, so as stated in Proposition \ref{Prop 3.5}, and then using the proposition   we get the following full sub-quiver of $\Gamma_{\CS(\La)}$
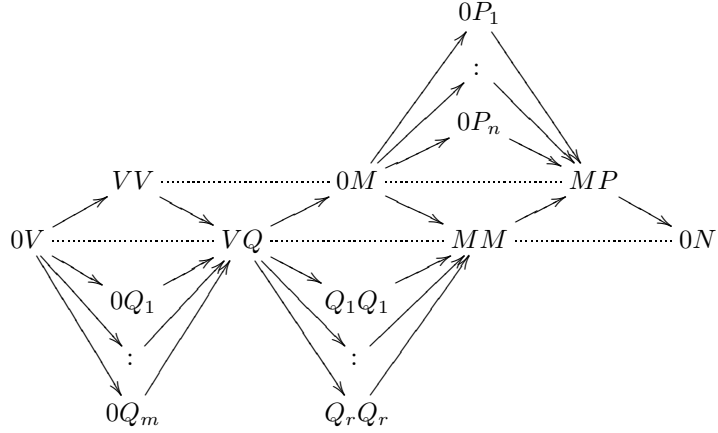
 \[
\xymatrix  @R=0.3cm  @C=0.6cm { 
&&&&&&& 0P_1\ar[rddd]&&&\\&&&& & && \colon\ar[rdd]&&&\\&&&&&&&0P_n\ar[dr]& &&	\\&&   &&VV\ar[dr]\ar@{.}[rr]&	&0M\ar[dr]\ar[ru]\ar[ruu]\ar[ruuu]\ar@{.}[rr]&&MP\ar[dr]&&&\\&&  &0V\ar[dr]\ar[ru]\ar[rddd]\ar[rdd]\ar@{.}[rr]&&
	VQ\ar[dr]\ar@{.}[rr]\ar[ru]\ar[rdd]\ar[rddd]&&MM\ar[ru]\ar@{.}[rr]&&0N&&\\
&&	&&0Q_1\ar[ur]&	&Q_1Q_1\ar[ru]&&&&&\\&&&&\colon\ar[ruu]&& \colon\ar[ruu]&&&&\\&&&& 0Q_m\ar[ruuu]&&Q_rQ_r\ar[ruuu]&&&&}
\]
But $m=r$, since $Q=\oplus^m_1Q_i= \oplus^r_1Q_i$. Iterating the same construction for $\left(\begin{smallmatrix} 0\\ V\end{smallmatrix}\right)_{0}$ and so on, we obtain the full sub-quiver of $\Gamma_{\CS(\rm{Gproj}\mbox{-}\La)}$ containing all  $\tau^i_{\CG}\left(\begin{smallmatrix} 0\\ N\end{smallmatrix}\right)_{0}$ for $i\geq 0$. 
 
Following the inverse construction, but this turn by applying Proposition \ref{Propo 3.7}, we get the following sub-quiver of $\Gamma_{\CS(\rm{Gproj}\mbox{-}\La)}$ for $M$ 
 
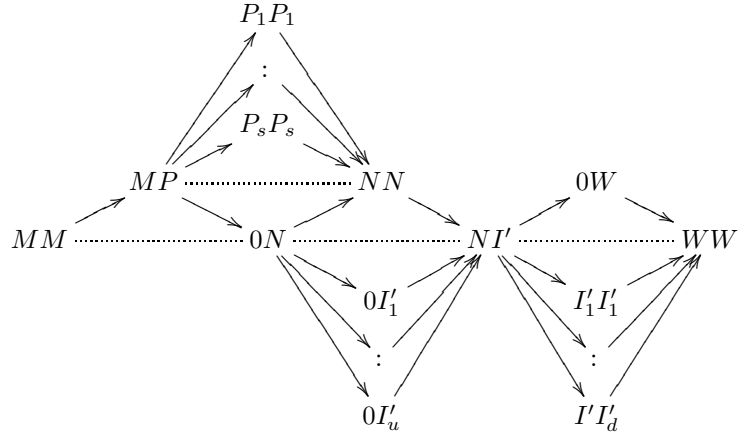
	\[
\xymatrix  @R=0.3cm  @C=0.6cm { 
	&&&&& P_1P_1\ar[rddd]&&&&&&\\	&&&&& \colon\ar[rdd]&&&&&&\\	&&&&&P_sP_s\ar[dr]& &&&&&	\\   &&&	&MP\ar[dr]\ar[ru]\ar[ruu]\ar[ruuu]\ar@{.}[rr]&&NN\ar[dr]&&0W\ar[rd]&&&&\\  &&&
	MM\ar@{.}[rr]\ar[ru]&&0N\ar[ru]\ar@{.}[rr]\ar[dr]\ar[rdd]\ar[rddd]&&NI'\ar[ru]\ar[rd]\ar[rdd]\ar[rddd]\ar@{.}[rr]&&WW&&&\\
	&&&	&&&0I'_1\ar[ur]&&I'_1I'_1\ar[ur]&&&&\\&&&&& &\colon\ar[ruu] &&\colon\ar[ruu]&&&\\&&&& &&0I'_u \ar[ruuu]&&I'I'_d \ar[ruuu]&&&}
\]
But $u=d$, since $I'=\oplus^u_1I'_i= \oplus^d_1I'_i$. Iterating the same construction for  $\left(\begin{smallmatrix} W\\ W\end{smallmatrix}\right)_{1}$ and so on, we find the full sub-quiver of $\Gamma_{\CS(\rm{Gproj}\mbox{-}\La)}$ containing  all $\tau^{i}_{\CG} \left(\begin{smallmatrix} 0\\ N\end{smallmatrix}\right)_{0}$ for $i\leqslant 0$. Having both above constructions together give us the part of the competent of  $\Gamma_{\CS(\rm{Gproj}\mbox{-}\La)}$ containing the $\tau_{\CG}$-orbit of $\left(\begin{smallmatrix} 0\\ N\end{smallmatrix}\right)_{0}$.  By deleting projective-injective vertices we obtain that the following component of the stable Auslander-Reiten quiver $\Gamma^s_{\CS(\rm{Gprj}\mbox{-}\La)}$: 
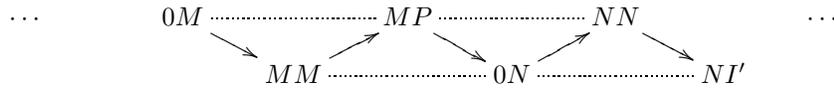
  \[
 \xymatrix  @R=0.3cm  @C=0.6cm { 
 	\cdots&&0M\ar[rd]\ar@{.}[rr]&	&MP\ar[dr]\ar@{.}[rr]&&NN\ar[dr]&&\cdots&\\&&&
 	MM\ar@{.}[rr]\ar[ru]&&0N\ar[ru]\ar@{.}[rr]&&NI'&&}
 \]
which stops in both left and right sides as $\La$ is CM-finite, by Proposition \ref{CM-finite}. Thus the underlying graph of such a component is a cycle. For instance, assume  the module $N$, as stated in Proposition \ref{Prop 3.5},  is of  $\tau_{\CG}$-period 2, i.e., $\tau^2_{\CG}N=N$, or equivalently of $\Omega_{\CG}$-period 2. Hence we have $V=N$. Based on our construction  the  corresponding components of $\Gamma^s_{\CS(\rm{Gprj}\mbox{-}\La)}$ containing $\left(\begin{smallmatrix} 0\\ N\end{smallmatrix}\right)_{0}$ is of the following form

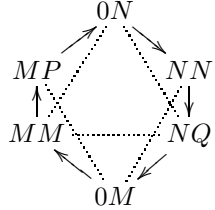
\[
\xymatrix@R1em@C.5em{&0N\ar@{<-}[dl]\ar@{->}[dr]\ar@{.}[ldd]\\ MP\ar@{<-}[d]\ar@{.}[rdd]&&NN\ar@{->}[d]\\ MM\ar@{<-}[dr]\ar@{.}[rr]&&NQ\ar@{->}[dl]\ar@{.}[luu]\\ &0M\ar@{.}[ruu]}\ \ \ 
\]

\end{construction}

 Let us formulate the above observation  for the case that $\La$ is an  $\Omega_{\CG}$-algebra.
 
\begin{definition} Denote $Z_n$ a basic $n$-cycle with the vertex set $\{1, 2,\cdots, n \}$
 and the arrow set $\{\alpha_1, \alpha_2, \cdots, \alpha_n\}$, where $s(\alpha_i)=i$  and $t(\alpha_i)=i+1$ for each $1\leqslant i \leqslant n$.  Here, we identify $n+1$ with 1. Let $I$ denote the arrow ideal of the path algebra  $kZ_n$ over a field $k$. We recall that $I$ generated by all arrows in $Z_n.$ Let $A_n=kZ_n/I^2$ be the corresponding quadratic  monomial algebra. The symbol $kZ_n$ denotes also the path category of $Z_n$ and whose ideal arrow is defined similarly.
 \end{definition}
 \begin{theorem}\label{Theorem 4.10}
 Let $\La$ be an $\Omega_{\CG}$-algebra. The following assertions hold.
 	
 \begin{itemize}
 	
 	\item [$(1)$] Any indecomposable object in $\CS(\rm{Gprj}\mbox{-}\La)$ is isomorphic to either $\left(\begin{smallmatrix} 0\\ G\end{smallmatrix}\right)_0$  or $\left(\begin{smallmatrix} G\\ G\end{smallmatrix}\right)_1$ or $\left(\begin{smallmatrix} \Omega_{\La}(G)\\ P\end{smallmatrix}\right)_i$, where $i$ the kerenl of the projective cover $P\rt G$, for some indecomposable non-projective Gorenstein projective module $G$. In particular, $\CS(\rm{Gprj}\mbox{-}\La)$ is of finite representation type, or equivalently $T_2(\La)$ is CM-finite.
 	\item [$(2)$] Let $G$ be an indecomposable non-projective module in $\text{Gprj}\mbox{-}\La$. Assume $n$ is the least positive integer such that $\Omega^n_{\La}(G)\simeq G$. Consider the following exact sequence  obtained  from the minimal projective resolution of $G$
 	$$0 \rt \Omega^n_{\La}(G)\st{p_n}\rt P^{n-1}\st{p_{n-1}}\rt \cdots\rt P_1\st{p_1}\rt P^0\st{p_0}\rt G\rt 0,$$
 and  for $ 0 \leqslant i \leqslant n-1$ denote $\epsilon_i: 0\rt K^i\st{j_i}\rt P^i\rt K^{i-1}\rt 0$ the induced short exact sequences by the above exact sequence, where $K^i$ is the kernel of $p_i$ (so by our convention $K^{n-1}=\Omega_{\La}(G)$ and set $K^{-1}=G$). Then the vertices of the component $\Gamma^s_{G}$ containing the vertex $\left(\begin{smallmatrix} 0\\ N\end{smallmatrix}\right)_{0}$ are corresponded to the following indecomposable objects 
 
   $$\left(\begin{smallmatrix} 0\\ K^i\end{smallmatrix}\right)_{0},  \ \ 
  \left(\begin{smallmatrix} K^i\\ K^i\end{smallmatrix}\right)_{1}, \ \ 
 \left(\begin{smallmatrix} K^i\\ P^i\end{smallmatrix}\right)_{j_i}$$ 	
 	where $i$ runs through $0\leqslant i \leqslant n-1.$ In particular, the component $\Gamma^s_{G}$ of $\Gamma^s_{\CS(\rm{Gprj}\mbox{-}\La)}$ has $3n$ vertices. 
 Further, the stable Auslander-Reiten quiver $\Gamma^s_{\CS(\rm{Gprj}\mbox{-}\La)}$ is a disjoint union of the all  components $\Gamma^s_G$, where $G$ belong to a  fixed set of   representatives of the all equivalence classes in $\CC(\La)$. In particular, there is a bijection between the components of $\Gamma^s_{\CS(\rm{Gprj}\mbox{-}\La)}$ and the elements of $\CC(\La)$.
 	\item[$(3)$] Assume  $\La$ is a finite dimensional algebra over an algebraic closed filed $k$ of  	characteristic different from two. Let $\Gamma^s_{G_1}, \cdots, \Gamma^s_{G_n}$ denote all the components, and denote by  $d_i$ the number of vertices in $\Gamma^s_{G_i}$ for each $1\leqslant i \leqslant n$.
 	  Then there is an (additive) equivalence  	
 	 $$\underline{\CS(\rm{Gprj}\mbox{-}\La)}\simeq kZ_{d_1}/I^2_1 \times \cdots \times kZ_{d_n}/I^2_n,$$
 	where $KZ_{d_i}/I^2_i$ is the the quotient category of the  path category $kZ_n$ modulo the ideal $I^2_i$, the square of the arrow ideal $I_i$, indeed,   generated by all paths  of length $\geqslant 2$. Moreover, if $\La$ is Gorenstein, then  $\mathbb{D}_{\text{sg}}(T_2(\La))\simeq  kZ_{d_1}/I^2_1 \times \cdots \times kZ_{d_n}/I^2_n$ (as additive categories).
 \end{itemize}
 
 \end{theorem}
\begin{proof}
$(1)$ Assume $\left(\begin{smallmatrix} G\\ G'\end{smallmatrix}\right)_{f}$ is an indecomposable object in $\CS(\rm{Gprj}\mbox{-}\La)$. It induces the short exact sequence $\lambda: 0 \rt G\st{f}\rt G'\rt \text{Cok}(f)\rt 0$ with all terms in $\text{Gprj}\mbox{-}\La.$ But Lemma \ref{Lemma1}  implies that the sequence $\lambda$ is isomorphic to 
the short exact sequences of the  form either $0 \rt \Omega_{\La}(G) \rt P \rt G \rt 0$ or $0 \rt 0 \rt G  \st{1} \rt G \rt 0$ or $0 \rt G \st{1} \rt G \rt 0 \rt 0$ 	for some indecomposable  module $G$ in $\rm{Gprj} \mbox{-} \La.$  Hence we get the first part. The second part follows from the first part and using these two facts that $\Omega_{\CG}$-algebras are $\text{CM}$-finite (by Proposition \ref{CM-finite}) and the operation of taking projective cover gives us only a finite number indecomposable objects  up to isomorphism in $\CS(\rm{Gprj}\mbox{-}\La)$.

$(2)$ Applying Construction \ref{Construction 4.8} for the two short exact sequences $\epsilon_0, \epsilon_1 $ (if $n=1$, applying for $\epsilon_0, \epsilon_0$) we reach the following part of $\Gamma_{\CS(\rm{Gprj}\mbox{-}\La)}$
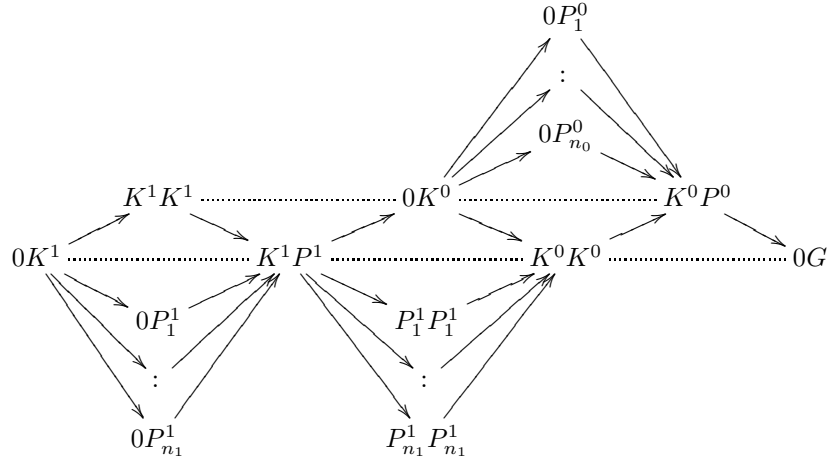
\[
\xymatrix  @R=0.3cm  @C=0.6cm { 
	&&&&&&& 0P^0_1\ar[rddd]&&&\\&&&& & && \colon\ar[rdd]&&&\\&&&&&&&0P^0_{n_0}\ar[dr]& &&	\\&&   &&K^1K^1\ar[dr]\ar@{.}[rr]&	&0K^0\ar[dr]\ar[ru]\ar[ruu]\ar[ruuu]\ar@{.}[rr]&&K^0P^0\ar[dr]&&&\\&&  &0K^1\ar[dr]\ar[ru]\ar[rddd]\ar[rdd]\ar@{.}[rr]&&
	K^1P^1\ar[dr]\ar@{.}[rr]\ar[ru]\ar[rdd]\ar[rddd]&&K^0K^0\ar[ru]\ar@{.}[rr]&&0G&&\\
	&&	&&0P^1_1\ar[ur]&	&P^1_1P^1_1\ar[ru]&&&&&\\&&&&\colon\ar[ruu]&& \colon\ar[ruu]&&&&\\&&&& 0P^1_{n_1}\ar[ruuu]&&P^1_{n_1}P^1_{n_1}\ar[ruuu]&&&&}
\]

Repeating the same construction for the pair of the short exact sequences $(\epsilon_1, \epsilon_2)$ until to the pair $(\epsilon_{n-2}, \epsilon_{n-1})$, we will obtain $n-1$ full subquivers of $\Gamma_{\CS(\rm{Gprj}\mbox{-}\La)}$ as the above such the one corresponding to $(\epsilon_{n-2}, \epsilon_{n-1})$ has the object $\left(\begin{smallmatrix} 0\\ G\end{smallmatrix}\right)_{0}$ in the leftmost side. Hence the construction will stop at $(n-1)$-th step. By glowing the obtained full subquivers we obtain the full subquiver $\tilde{\Gamma}_G$ of $\Gamma_{\CS(\rm{Gprj}\mbox{-}\La)}$ containing the $\tau_{\CG}$-orbit of $\left(\begin{smallmatrix} 0\\ N\end{smallmatrix}\right)_{0}$.  Moreover, in view of the statement $(1)$ we see that $\Gamma_{\CS(\rm{Gprj}\mbox{-}\La)}$ is a union of all full subquiver (not necessarily disjoint) $\Gamma_{\CS(\rm{Gprj}\mbox{-}\La)}$, where  $G$ runs through a  fixed set of   representatives of the all equivalence classes in $\CC(\La)$.   Be deleting the projective-injective vertices of  the full subquiver $\tilde{\Gamma}_G$ as we obtained in the above, then  we obtain the component $\Gamma_G,$ as presented  in the below

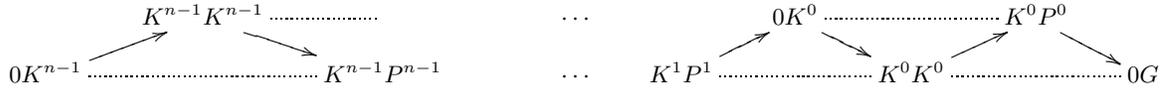
\small{\[
\xymatrix  @R=0.3cm  @C=0.6cm { 
	&K^{n-1}K^{n-1}\ar[rd]\ar@{.}[r]&&&\cdots&	&0K^0\ar[dr]\ar@{.}[rr]&&K^0P^0\ar[dr]&\\0K^{n-1}\ar[ru]\ar@{.}[rr]&&K^{n-1}P^{n-1}&&\cdots&
	K^1P^1\ar@{.}[rr]\ar[ru]&&K^0K^0\ar[ru]\ar@{.}[rr]&&0G}\]}
We have to identify the vertex corresponding to $0K^{n-1}$ with to $0G$ as by our assumption $K^{n-1}=\Omega_{\La}(G)\simeq G.$  Now our construction reveals the all facts in the statement $(2)$.

$(3)$ Our assumption on $\La$ guarantees that the components of $\Gamma^s_{\CS(\rm{Gprj}\mbox{-}\La)}$  must be standard. Hence  $\underline{\CS(\rm{Gprj}\mbox{-}\La)}$ is equivalent to a product of the mesh categories of all the components. On the other hand, the mesh category of the component $\Gamma^s_{G_i}$ is equivalent to the quotient category $kZ_{d_i}/I^2_i$, so the result. When $\La$ is Gorenstein, so is $T_2(\La)$ \cite[Corollary 4.3]{AHKe1}. Then the last part follows from a known result of Buchweitz which implies  $\mathbb{D}_{\rm{sg}}(T_2(\La))\simeq \underline{\CS(\rm{Gprj}\mbox{-}\La)}$. 
\end{proof} 
 
\begin{remark}
 When $\La$ is a  Gorenstein finite dimensional algebra over an algebraic closed filed $k$ of  	characteristic different from two,  as mentioned in the proof of the above theorem $T_2(\La)$ so is. Hence by applying Happel's work, $\underline{\CS(\rm{Gprj}\mbox{-}\La)}$ gets a triangulated structure. Thus in the equivalence given in the third part of the theorem  the left side  is a triangulated category. We deduce by such  an equivalence any path quotient category $kZ_{3n}/I^2$ can be enriched by  a triangulated structure. It might be interesting to study the induced triangulated structure of  $kZ_{3n}/I^2$. Viewing all the quotient categories $KZ_{d_i}/I^2_i$ with the induced triangulated  structure, then the additive equivalence in theorem turns into a triangle equivalence. 
\end{remark}
By our interpretation  in Theorem \ref{Theorem 4.10} for  $\CC(\La)$ in terms of the components of $\Gamma^s_{\CS(\rm{Gprj}\mbox{-}\La)}$ we prove in the next result that the number of the equivalence classes in $\CC(\La)$ is a derived invariant for $\Omega_{\CG}$-algebra. It is a generalization of \cite[Corollary 3.1]{Ka}. 
\begin{corollary}
	Let $\La$ and $\La'$ be $\Omega_{\CG}$-algebras (not necessarily finite dimensional algebras). If there is an equivalence of triangulated
	categories $\mathbb{D}^{\rm{b}}(\mmod \La)\simeq \mathbb{D}^{\rm{b}}(\mmod \La')$, then there is a bijection of sets 
	$$f:\CC(\La)\rt \CC(\La')$$
	such that for each  indecomposable non-projective  module $G \in \text{Gprj}\mbox{-}\La,$ $l(G)=l(G')$, where $G' \in \text{Gprj}\mbox{-}\La'$ and $[G']=f([G])$.
\end{corollary}
\begin{proof}
	In view of \cite[Theorem 8.5]{A}   $\mathbb{D}^{\rm{b}}(\mmod \La)\simeq \mathbb{D}^{\rm{b}}(\mmod \La')$ gives the derived equivalence $\mathbb{D}^{\rm{b}}(\mmod T_2(\La))\simeq \mathbb{D}^{\rm{b}}(\mmod T_2(\La'))$. In fact, we specialize \cite[Theorem 8.5]{A} for the case that the quiver is $\bullet\rt \bullet.$ By \cite[Theorem 4.1.2]{AHV}, stating that if two Gorenstein  Artin algebras $A$ and $A'$ are derived equivalent, then there is a triangulated equivalence $\text{Gprj}\mbox{-}A\simeq \text{Gprj}\mbox{-}A'$, then we obtain
	$\underline{\CS(\rm{Gprj}\mbox{-}A)}\simeq \underline{\CS(\rm{Gprj}\mbox{-}A')}$ by applying the theorem of \cite{AHV} for $T_2(\La)$ and $T_2(\La')$. Now Theorem \ref{Theorem 4.10} completes the proof.
\end{proof}

\section{The morphism category over $\Omega_{\mathcal{G}}$-algebras}\label{Section 4}
 This section is devoted to study of the morphism categories $\text{H}(\text{Gprj}\mbox{-}\La)$ over $\Omega_{\CG}$-algebras. We will compute some certain almost split sequences in $\text{H}(\text{Gprj}\mbox{-}\La)$. It seems in contrast to $\CS(\text{Gprj}\mbox{-}\La)$ it is difficult to compute all the almost split sequences. However, in the end of section for a very special case of $\Omega_{\CG}$-algebras we will completely compute all the almost split sequences.
\begin{proposition}
	Let $\La$ be an Artin algebra such that the  subcategory of Gorenstein projective modules $\rm{Gprj}\mbox{-}\La$ is contravariantly finite. Then $\rm{H}(\rm{Gprj}\mbox{-}\La)$ is functorially finite in $\rm{H}(\La).$ In particular, $\rm{H}(\rm{Gprj}\mbox{-}\La)$ has almost split sequences.
\end{proposition}
\begin{proof}
	Thanks to \cite{KS} it suffices to show that $\rm{H}(\rm{Gprj}\mbox{-}\La)$ is contravariantly finite in $\rm{H}(\La)$. Assume an arbitrary object $\left(\begin{smallmatrix} X\\ Y\end{smallmatrix}\right)_{f}$ of $\rm{H}(\La)$ is given. Take a right $\rm{Gprj}\mbox{-}\La$-approximation $\alpha: G\rt Y$ in $\mmod\La.$ Let $(H, a, b)$ be a pull-back of $(\alpha, f)$ and $\beta:G'\rt H$ a right $\rm{Gprj}\mbox{-}\La$-approximation of $H$ in $\mmod \La.$	The below diagram contains all  the data we have already defined
\small{	\begin{equation*}
	\begin{tikzcd}
	G' \ar[dotted]{dd}{a\beta} \ar{rr}{b\beta} \drar{\beta} && X \ar{dd}{f} \\
	& H \urar{b} \dlar{a} \\
	G\ar{rr}{\alpha} && Y
	\end{tikzcd}
	\end{equation*}}
	
	By the above diagram we get the morphism $\left(\begin{smallmatrix} b\beta \\ \alpha \end{smallmatrix}\right):\left(\begin{smallmatrix} G' \\ G\end{smallmatrix}\right)_{\alpha\beta}\rt \left(\begin{smallmatrix} X \\ Y\end{smallmatrix}\right)_f$ such that $\left(\begin{smallmatrix} G' \\  G\end{smallmatrix}\right)_{\alpha\beta}$ lies in $\rm{H}(\rm{Gprj}\mbox{-}\La)$. Using the universal property of pull-back, one can easily prove that the obtained morphism in $\rm{H}(\La)$ is a right $\rm{H}(\rm{Gprj}\mbox{-}\La)$-approximation of $\left(\begin{smallmatrix} X \\ Y\end{smallmatrix}\right)_f$ in $\rm{H}(\La)$.	 
\end{proof}
We point out that as one can easily see that the proof of $\rm{H}(\rm{Gprj}\mbox{-}\La)$ being contravariantly finite works to show that $\rm{H}(\mathcal{D})$ is contravariantly finite in $\rm{H}(\La)$, whenever $\mathcal{D}$ is contravariantly finite in $\mmod \La$.

\begin{lemma}\label{Lemma 5.2}	
	Assume  $\delta:  0 \rt  A\st{f} \rt B\st{g} \rt C \rt 0$ is  an almost split sequence in $\CM.$ Then
	\begin{itemize}
		\item[$(1)$] The almost split sequence in $\rm{H}(\CM)$ ending at $\left(\begin{smallmatrix}0\\ C\end{smallmatrix}\right)_0$ has of the form 
		$$\xymatrix@1{  0\ar[r] & {\left(\begin{smallmatrix} A\\ A\end{smallmatrix}\right)}_{1}
			\ar[rr]^-{\left(\begin{smallmatrix} 1 \\ f\end{smallmatrix}\right)}
			& & {\left(\begin{smallmatrix}A\\ B\end{smallmatrix}\right)}_{f}\ar[rr]^-{\left(\begin{smallmatrix} 0 \\ g\end{smallmatrix}\right)}& &
			{\left(\begin{smallmatrix}0\\ C\end{smallmatrix}\right)}_{0}\ar[r]& 0. } \ \    $$		
		
		\item [$(2)$] The almost split sequence in $\rm{H}(\CM)$ ending at $\left(\begin{smallmatrix}C\\ C\end{smallmatrix}\right)_1$ has of  the form 
		
		$$\xymatrix@1{  0\ar[r] & {\left(\begin{smallmatrix} A\\ 0\end{smallmatrix}\right)}_{0}
			\ar[rr]^-{\left(\begin{smallmatrix} f \\ 0 \end{smallmatrix}\right)}
			& & {\left(\begin{smallmatrix} B\\ C\end{smallmatrix}\right)}_{g}\ar[rr]^-{\left(\begin{smallmatrix} g \\ 1\end{smallmatrix}\right)}& &
			{\left(\begin{smallmatrix}C \\ C\end{smallmatrix}\right)}_{1}\ar[r]& 0. } \ \    $$
		
	\end{itemize}
\end{lemma}
\begin{proof}
	The proof is motivated for the case $\CC=\mmod \La.$ We refer to  \cite[Proposition 3.1]{MO} or also \cite[Lemma 5.3]{HMa}. For convenience of the reader we provide a proof for the case $\CM=\rm{Gprj}\mbox{-}\La.$ 

For $(1)$,  let $\left(\begin{smallmatrix} 0 \\ d\end{smallmatrix}\right):\left(\begin{smallmatrix} X \\ Y\end{smallmatrix}\right)_h\rt \left(\begin{smallmatrix} 0 \\ C\end{smallmatrix}\right)_0$  be a non-retraction. Then the morphism $\left(\begin{smallmatrix} 0 \\ d\end{smallmatrix}\right):\left(\begin{smallmatrix} 0 \\ Y\end{smallmatrix}\right)_0\rt \left(\begin{smallmatrix} 0 \\ C\end{smallmatrix}\right)_0$ in $\CS(\rm{Gprj}\mbox{-}\La)$ is a non-retraction as well. Because of  Lemma \ref{AlmostSplittrivialmonomorphisms}(1)
we know that the short exact sequence in the statement is almost split in $\CS(\rm{Gprj}\mbox{-}\La)$, hence there exists $\left(\begin{smallmatrix} 0 \\ t\end{smallmatrix}\right):\left(\begin{smallmatrix} 0 \\ Y\end{smallmatrix}\right)_0\rt \left(\begin{smallmatrix} A \\ B\end{smallmatrix}\right)_f$  such that $\left(\begin{smallmatrix} 0 \\ d\end{smallmatrix}\right)=\left(\begin{smallmatrix} 0 \\ g\end{smallmatrix}\right)\circ \left(\begin{smallmatrix} 0 \\ t\end{smallmatrix}\right)$. Since $gth=0$, there is a morphism $s:X\rt A$ with $fs=th.$ Now we see that the morphism $\left(\begin{smallmatrix} 0 \\ d\end{smallmatrix}\right)$    factors
through $\left(\begin{smallmatrix} 0 \\ g\end{smallmatrix}\right)$ via $\left(\begin{smallmatrix} s \\ t\end{smallmatrix}\right)$. So we are done this case.

For $(2),$	since $g:B\rt C$  does not split, the map $\left(\begin{smallmatrix} g \\ 1\end{smallmatrix}\right):\left(\begin{smallmatrix} B\\ C\end{smallmatrix}\right)_g\rt \left(\begin{smallmatrix} C \\ C\end{smallmatrix}\right)_1$ so does not. Let $\left(\begin{smallmatrix} v \\ w\end{smallmatrix}\right):\left(\begin{smallmatrix} X \\ Y\end{smallmatrix}\right)_h\rt \left(\begin{smallmatrix} C \\ C\end{smallmatrix}\right)_1$ be a map which  is not a splittable epimorphism. Then $wh=v$. We claim $v$ is not a splittable epimorphism. Indeed, if $v$ is a splittable epimorphism, then there exists a morphism $s:C\rt X$   such that $vs=1$. Then we get the following commutative diagram
$$\xymatrix{		   C \ar[d]^s
	\ar@{=}[r] & C\ar[d]^{hs}  \\
	X
	\ar[d]^v \ar[r]^h & Y \ar[d]^w  \\
	C
	\ar@{=}[r] & C   }	$$
Hence $\left(\begin{smallmatrix} v \\ w\end{smallmatrix}\right)$ is a splittable epimorphism, a contradiction. As $g:B\rt C$ is a right almost split morphism, there exists a map $r:X\rt B$  such that $gr=v$. Now we can make  the following commutative diagram:
$$\xymatrix{		   X \ar[d]^r
	\ar[r]^h & Y\ar[d]^w  \\
	B
	\ar[d]^g \ar[r]^g & C \ar@{=}[d]  \\
	C
	\ar@{=}[r] & C   }	$$

Not that as we mention in the above $wh=v.$ The above digram says that the morphism $\left(\begin{smallmatrix} v\\ w\end{smallmatrix}\right)$ can be factored through $\left(\begin{smallmatrix} g \\ 1\end{smallmatrix}\right)$. Hence $\left(\begin{smallmatrix} g \\ 1\end{smallmatrix}\right)$ is a right almost morphism in $\rm{H}(\rm{Gprj}\mbox{-}\La)$. This ends the proof.	
\end{proof}

\begin{proposition}\label{Prop 5.3}
Let $\La$ be an $\Omega_{\CG}$-algebra. Assume $A$ is an indecomposable non-projective module in $\rm{Gprj}\mbox{-}\La$. Consider the short exact sequences $0\rt \Omega^2_{\La}(A)\st{i_1}\rt Q\st{p_1}\rt \Omega_{\La}(A)\rt 0$ and $0 \rt \Omega_{\La}(A)\st{i_0}\rt P\st{p_0}\rt A\rt 0$ (induced by the minimal projective resolution of $A$, and so almost split sequence in $\rm{Gprj}\mbox{-}\La$). The following short exact sequences in which the morphisms defined in an obvious way (see the proof below) is an almost split sequence in $\rm{H}(\rm{Gprj}\mbox{-}\La)$

	$$\xymatrix@1{  0\ar[r] & {\left(\begin{smallmatrix} Q\\ \Omega_{\La}(A)\end{smallmatrix}\right)_{p_1}}
	\ar[rr]
	& & {\left(\begin{smallmatrix}Q\\ P\end{smallmatrix}\right)_{i_0p_1}}\oplus {\left(\begin{smallmatrix}\Omega_{\La}(A)\\ \Omega_{\La}(A)\end{smallmatrix}\right)_1} \ar[rr]& &
	{\left(\begin{smallmatrix}\Omega_{\La}(A)\\ P\end{smallmatrix}\right)_{i_0}}\ar[r]& 0. } \ \    $$
Further, the object $\left(\begin{smallmatrix}Q\\ P\end{smallmatrix}\right)_{i_0p_1}$ is indecomposable.
\end{proposition}
\begin{proof}
	According to Lemma \ref{Lemma 5.2} we have the following mesh diagrams in $\Gamma_{\rm{H}(\text{Gprj}\mbox{-}\La)}$  \[
	\xymatrix  @R=0.3cm  @C=0.6cm { 
			&Q\Omega_{\La}(A)\ar[dr]\ar@{.}[rr]&&\Omega_{\La}(A)P\ar[dr]&\\
		\Omega^2_{\La}(A)0\ar@{.}[rr]\ar[ru]&&\Omega_{\La}(A)\Omega_{\La}(A)\ar[ru]\ar@{.}[rr]&&0A}
	\]
By the above diagram  we conclude $\tau_{\text{H}(\text{Gprj}\mbox{-}\La)}\left(\begin{smallmatrix} \Omega_{\La}(A)\\ P\end{smallmatrix}\right)_{i_0}=\left(\begin{smallmatrix} Q\\ \Omega_{\La}(A)\end{smallmatrix}\right)_{p_1}$. On the other hand, we have the following sequence in $\text{H}(\text{Gprj}\mbox{-}\La)$, the same as  the one  in the statement, 
$$\xymatrix@1{  0\ar[r] & {\left(\begin{smallmatrix} Q\\ \Omega_{\La}(A)\end{smallmatrix}\right)_{p_1}}
	\ar[rr]^>>>>>>>>>{\left(\begin{smallmatrix}\phi_1\\ \phi_2\end{smallmatrix}\right)}
	& & {\left(\begin{smallmatrix}Q\\ P\end{smallmatrix}\right)_{i_0p_1}}\oplus {\left(\begin{smallmatrix}\Omega_{\La}(A)\\ \Omega_{\La}(A)\end{smallmatrix}\right)_1} \ar[rr]^>>>>>>>>{\left(\begin{smallmatrix}\psi_1 & \psi_2\end{smallmatrix}\right)}& &
	{\left(\begin{smallmatrix}\Omega_{\La}(A)\\ P\end{smallmatrix}\right)_{i_0}}\ar[r]& 0, } \ \    $$ 
where $\phi_1=\left(\begin{smallmatrix}1\\ i_0\end{smallmatrix}\right):\left(\begin{smallmatrix}Q\\ \Omega_{\La}(A)\end{smallmatrix}\right)_{p_1}\rt \left(\begin{smallmatrix}p_1\\ P\end{smallmatrix}\right)_{i_0p_1}$, $\phi_2=\left(\begin{smallmatrix}p_1\\ 1\end{smallmatrix}\right):\left(\begin{smallmatrix}Q\\ \Omega_{\La}(A)\end{smallmatrix}\right)_{p_1}\rt \left(\begin{smallmatrix}\Omega_{\La}(A)\\ \Omega_{\La}(A)\end{smallmatrix}\right)_{1}$, $\psi_1=\left(\begin{smallmatrix}p_1\\ 1\end{smallmatrix}\right):\left(\begin{smallmatrix}Q\\ P\end{smallmatrix}\right)_{i_0p_1}\rt \left(\begin{smallmatrix}\Omega_{\La}(A)\\ P\end{smallmatrix}\right)_{i_0}$ and $\psi_2=\left(\begin{smallmatrix}-1\\ -i_0\end{smallmatrix}\right):\left(\begin{smallmatrix}\Omega_{\La}(A)\\ \Omega_{\La}(A)\end{smallmatrix}\right)_{1}\rt \left(\begin{smallmatrix}\Omega_{\La}(A)\\ P\end{smallmatrix}\right)_{i_0}$. To complete the proof in view of this fact that the left end term of the above non-split sequence is the Auslander-Reiten translation of the right end term, it suffices to prove that any non-isomorphism $\left(\begin{smallmatrix}a\\ b\end{smallmatrix}\right):\left(\begin{smallmatrix}Q\\ \Omega_{\La}(A)\end{smallmatrix}\right)_{p_1}\rt \left(\begin{smallmatrix}Q\\ \Omega_{\La}(A)\end{smallmatrix}\right)_{p_1}$ factors through $\left(\begin{smallmatrix}Q\\ \Omega_{\La}(A)\end{smallmatrix}\right)_{p_1}$. Since $p_1$ is a projective cover, the morphism $b$ is impossible to be an isomorphism. Otherwise, $a$ is also an isomorphism, so this contradicts this fact that $\left(\begin{smallmatrix}a\\ b\end{smallmatrix}\right)$ is not an isomorphism. This implies that $b$ factors through $i_0$ via some morphism, say $d:P\rt \Omega_{\La}(A)$, as by our assumption $i_0$ is  left minimal almost split. It is easy to see that the morphism from the middle term of the above short exact sequence such whose restrictions to $\left(\begin{smallmatrix}Q\\ P\end{smallmatrix}\right)_{i_0p_1}$ and $\left(\begin{smallmatrix}\Omega_{\La}(A)\\ \Omega_{\La}(A)\end{smallmatrix}\right)_{1}$ are $\left(\begin{smallmatrix}a\\ d\end{smallmatrix}\right)$ and the zero map, respectively. The following commutative diagram verifies the introduced morphism is well-defined in the morphism category.
$$\xymatrix@R=0.3cm@C=0.5cm{ & Q\ar[dd]^{p_1}\ar[dl]_{a}\ar[rrr]^1& & &  Q\ar[dd]^{i_0p_1} \ar[dllll]_a\\
	Q\ar[dd]^{p_1} & & &\\
	& \Omega_{\La}(A)\ar[rrr]^{i_0}\ar[dl]^>>>>>>b &
	& &P\ar[dllll]^d\\
	\Omega_{\La}(A) 
	& & & &}$$
 For the last statement, we can decompose as $\left(\begin{smallmatrix}Q\\ P\end{smallmatrix}\right)_{i_0p_1}=\left(\begin{smallmatrix}X\\ Y\end{smallmatrix}\right)_{f}\oplus \left(\begin{smallmatrix}Z\\ 0\end{smallmatrix}\right)_{0}\oplus\left(\begin{smallmatrix}0\\ W\end{smallmatrix}\right)_{0} $  with $f\neq 0$. Because of being projective of $P$ and $Q$ we get $Z$ and $W$ must be projective. Since $\text{Im}(f)\simeq \Omega_{\La}(A)$ and $\Omega_{\La}(A)$ is indecomposable, we can deduce that $\left(\begin{smallmatrix}X\\ Y\end{smallmatrix}\right)_{f}$ is indecomposable. But $Z=0$. Otherwise, it contradicts the fact that $p_1$ is a projective cover. We also have $W=0$. Otherwise $A\simeq Z$, it means that $A$ would be projective. It is absurd. Now the proof is completed.
\end{proof}
The above result establishes some connection between certain simple modules over the Cohen-Macaulay Auslander algebras of $\Omega_{\CG}$-algebras.
\begin{corollary}\label{Cor 4.4}
Let $\La$ be an $\Omega_{\CG}$-algebra. Denote by $B$ the  Cohen-Macaulay Auslander algebra of $\La$. If $S$ a simple module in $\mmod B$ with the projective dimension 2, then $\tau^{-1}_B S\simeq \Omega_{B}(S')$, where $S'$ is so  a simple module in $\mmod A$ with the projective dimension 2. In particular, $\Ext^2_B(S, S')\simeq D\Hom_{B}(S', S)$.
\end{corollary}
\begin{proof}
	Let $G$ be a basic additive generator for $\text{Gprj}\mbox{-}\La$. As mentioned in the introduction the evaluation functor $\zeta_G$ induces an equivalence $\mmod (\text{Gprj}\mbox{-}\La)\simeq \mmod B.$ For a simple $B$-module $S$ as stated in the corollary, there exists an indecomposable non-projective module  $A$ in $\text{Gprj}\mbox{-}\La$ such that the functor $(-, A)/\text{rad}(-, A)$ is mapped into  $S$ by $\zeta_G$.    Moreover, we have the following minimal projective resolution 
	
	$$0 \rt (-, M)\st{(-, f)} \rt (-, N)\st{(-, g)}\rt (-, A)\rt (-, A)/\text{rad}(-, A)\rt 0,$$
	where $0 \rt M\st{f}\rt N\st{g  } \rt A\rt 0$ is an almost split sequence in $\text{Gprj}\mbox{-}\La.$ See \cite{ARVI} for a proof of the facts.  But, by Proposition \ref{Omega-algebra}, we may assume that the almost split sequence is equal to $\delta_1:0 \rt \Omega_{\La}(A)\st{i_0}\rt P\st{p_0}\rt A\rt 0.$
	
	Under the equivalence $\mmod (\text{Gprj}\mbox{-}\La)\simeq \mmod B$ we show  the statement for the functor $S:= (-, A)/\text{rad}(-, A).$ Applying Proposition \ref{Prop 5.3} for the short exact sequences $\delta_1$ and $\delta_2: 0\rt  A \st{i_{-1}}\rt Q\st{p_{-1}}\rt \Sigma_{\La} (A)\rt 0$ (which exists as $\text{Gprj}\mbox{-}\La$ is Frobenius) follows the following 
	
		$$\xymatrix@1{ \eta: \ \  0\ar[r] & {\left(\begin{smallmatrix} P\\ A\end{smallmatrix}\right)_{p_0}}
		\ar[rr]
		& & {\left(\begin{smallmatrix}P\\ Q\end{smallmatrix}\right)_{i_{-1}p_0}}\oplus {\left(\begin{smallmatrix}A\\ A\end{smallmatrix}\right)_1} \ar[rr]& &
		{\left(\begin{smallmatrix}A\\ Q\end{smallmatrix}\right)_{i_{-1}}}\ar[r]& 0. } \ \    $$
	is an almost split sequence in $\text{H}(\text{Gprj}\mbox{-}\La)$. We know from the introduction (Section \ref{Section 1}) the functor $\Phi:\text{H}(\text{Gprj}\mbox{-}\La)\rt \mmod (\text{Gprj}\mbox{-}\La)$. We apply the functor $\Phi$ on the above almost split sequence. It gives the following short exact sequence in $\mmod (\text{Gprj}\mbox{-}\La)$
	$$0 \rt \Phi\left(\begin{smallmatrix} P\\ A\end{smallmatrix}\right)_{p_0}\rt \Phi\left(\begin{smallmatrix}P\\ Q\end{smallmatrix}\right)_{i_{-1}p_0}\rt \Phi\left(\begin{smallmatrix}A\\ Q\end{smallmatrix}\right)_{i_{-1}}\rt 0    $$
	Note that $\Phi\left(\begin{smallmatrix}A\\ A\end{smallmatrix}\right)_{1}=0.$ Using this fact that sequence $\eta$ is  almost split in $\text{H}(\text{Gprj}\mbox{-}\La)$, then it is straightforward to check that $\Psi(\eta)$ is an almost split sequence in $\mmod (\text{Gprj}\mbox{-}\La)$. See also \cite[Proposition 5.7]{HMa} for a proof. By definition, $S=\Phi\left(\begin{smallmatrix} P\\ A\end{smallmatrix}\right)_{p_0}$, and by the almost split sequence $\Phi(\eta)$, $\tau^{-1}_B S=\Phi\left(\begin{smallmatrix} P\\ A\end{smallmatrix}\right)_{i_{-1}}$, remember here we applied the identification $\mmod (\text{Gprj}\mbox{-}\La)$ and $\mmod B$. Take $S'=(-, \Sigma_{\La}(A))/\text{rad}(-, A)$. Applying the Yoneda functor on the sequence  $\delta_2$ yields
	$$\xymatrix{		 
		(-, A)\ar[r]& (-,Q)\ar[rr]\ar[rd]&&(-, \Sigma_{\La}(A))\ar[rr]& & S'
		\ar[r] & 0  \\
		&&\tau^{-1}_BS\ar[ur]&&	& 
		&   }
	$$
The above diagram completes the first part.
For the second part, it is enough to use the Auslander-Reiten duality.			
\end{proof}

It might be a big and hard project to compute all the almost split sequences in $\rm{H}(\text{Gprj}\mbox{-}\La)$. However, we will do it for a very special case when $\La=A_n=kZ_n/I^2$. As $A_n$ is a Nakayama algebra with loewy length two, by  it\cite[Theorem 10.7]{AR1}, it is an $\Omega_{\CG}$-algebra. Also, it is  self-injective, hence $\mmod A_n=\text{Gprj}\mbox{-}A_n.$ 

For our computation we need the following general result.
\begin{proposition}\label{Proposition 5.5}
 Let $\La$ be a self-injective algebra. Assume   $ 0 \rt \Omega_{\La}(M)\st{i_0}\rt P^0\st{p_0}\rt M\rt 0$ and  $0 \rt \Omega^2_{\La}(M)\st{i_1}\rt P^1\st{p_1}\rt \Omega_{\La}(M)\rt 0$  are  almost split sequences in $\mmod \La$.  The following short exact sequences in which the morphisms defined in an obvious way  are   almost split sequences in $\rm{H}(\La)$

\begin{itemize}
	\item [$(a)$]
	
	$$\xymatrix@1{  0\ar[r] & {\left(\begin{smallmatrix} \Omega_{\La}(M)\\P^0 \end{smallmatrix}\right)_{i_0}}
		\ar[rr]
		& & {\left(\begin{smallmatrix}0\\ M\end{smallmatrix}\right)_{0}}\oplus {\left(\begin{smallmatrix}P^0\\ P^0\end{smallmatrix}\right)_1}\oplus {\left(\begin{smallmatrix}\Omega_{\La}(M)\\ 0\end{smallmatrix}\right)_{0}} \ar[rr]& &
		{\left(\begin{smallmatrix}P^0\\ M\end{smallmatrix}\right)_{p_0}}\ar[r]& 0. } \ \    $$

\item [$(b)$]
$$\xymatrix@1{  0\ar[r] & {\left(\begin{smallmatrix} P^1\\P^0 \end{smallmatrix}\right)_{i_0p_1}}
	\ar[rr]
	& & {\left(\begin{smallmatrix}P^1\\ 0\end{smallmatrix}\right)_{0}}\oplus {\left(\begin{smallmatrix}\Omega_{\La}(M)\\ P^0\end{smallmatrix}\right)_{i_0}} \ar[rr]& &
	{\left(\begin{smallmatrix}\Omega_{\La}(M)\\ 0\end{smallmatrix}\right)_{0}}\ar[r]& 0. } \ \    $$

	\item [$(c)$]

$$\xymatrix@1{  0\ar[r] & {\left(\begin{smallmatrix} 0\\\Omega_{\La}(M) \end{smallmatrix}\right)_{0}}
	\ar[rr]
	& & {\left(\begin{smallmatrix}0\\ P^0\end{smallmatrix}\right)_{0}}\oplus {\left(\begin{smallmatrix}P^1\\ \Omega_{\La}(M)\end{smallmatrix}\right)_{p_1}} \ar[rr]& &
	{\left(\begin{smallmatrix}P^1\\ P^0\end{smallmatrix}\right)_{i_0p_1}}\ar[r]& 0. } \ \    $$

\item [$(e)$]
$$\xymatrix@1{  0\ar[r] & {\left(\begin{smallmatrix} 0\\P^0 \end{smallmatrix}\right)_{0}}
	\ar[rr]
	& &  {\left(\begin{smallmatrix}P^1\\ P^0\end{smallmatrix}\right)_{i_0p_1}} \ar[rr]& &
	{\left(\begin{smallmatrix}P^1\\ 0\end{smallmatrix}\right)_{0}}\ar[r]& 0. } \ \    $$

\end{itemize}

\end{proposition}

\begin{proof}
Set first throughout the proof  $\tau_{H}:=\tau_{\text{H}(\La)}$.

$(a)$ We mention the assumption of $P^0$ being injective implies that $i_0$ is an injective envelop of $\Omega_{\La}(\La)$. Then by \cite[Proposition 4.4]{H3} we obtain $\tau_{H}\left(\begin{smallmatrix}P^0\\ M\end{smallmatrix}\right)_{p_0}=\left(\begin{smallmatrix}\Omega_{\La}(M)\\ P^0\end{smallmatrix}\right)_{i_0}$. Hence to show that the sequence in $(a)$ is  almost split it suffices to prove that any non-isomorphism $\left(\begin{smallmatrix} \alpha_1 \\ \alpha_2\end{smallmatrix}\right):\left(\begin{smallmatrix} \Omega_{\La}(M) \\ P^0\end{smallmatrix}\right)_{i_0}\rt \left(\begin{smallmatrix} \Omega_{\La}(M) \\ P^0\end{smallmatrix}\right)_{i_0}$ factors through the monomorphism lying in the sequence. It is not true $\alpha_1$ to be an isomorphism. Otherwise, $\alpha_1$ would be an isomorphism, as $i_0$ is an injective envelop, this means that $\left(\begin{smallmatrix} \alpha_1 \\ \alpha_2\end{smallmatrix}\right)$ is an isomorphism, a contradiction. Because  of being $i_0$ a left minimal almost split, there is a morphism $\beta:P^0\rt \Omega_{\La}(M)$ such that $\beta i_0=\alpha_1.$   As $(\alpha_2-i_0\beta)i_0=0$, so there is $\eta:M\rt P^0$ such that $\eta p_0=\alpha_2-i_0\beta.$ Define the map from the middle term of the sequence in  $(a)$ to $\left(\begin{smallmatrix} \Omega_{\La}(M) \\ P^0\end{smallmatrix}\right)_{i_0}$ such its restriction on the direct summands ${\left(\begin{smallmatrix}0\\ M\end{smallmatrix}\right)_{0}},  {\left(\begin{smallmatrix}P^0\\ P^0\end{smallmatrix}\right)_1}$ and $ {\left(\begin{smallmatrix}\Omega_{\La}(M)\\ 0\end{smallmatrix}\right)_{0}}$ are ${\left(\begin{smallmatrix}0\\ \eta\end{smallmatrix}\right)}, {\left(\begin{smallmatrix}\beta \\ i_0\beta \end{smallmatrix}\right)}$ and ${\left(\begin{smallmatrix}0\\ 0\end{smallmatrix}\right)}$, respectively. One can easily check that the defined map gives  the desired factorization.

 We have by Proposition \ref{Prop 5.3} the following almost split sequence in $\text{H}(\La)$
	$$\xymatrix@1{ (f) \ \ \  0\ar[r] & {\left(\begin{smallmatrix} P_1\\ \Omega_{\La}(M)\end{smallmatrix}\right)_{p_1}}
	\ar[rr]
	& & {\left(\begin{smallmatrix}P_1\\ P^0\end{smallmatrix}\right)_{i_0p_1}}\oplus {\left(\begin{smallmatrix}\Omega_{\La}(M)\\ \Omega_{\La}(M)\end{smallmatrix}\right)_1} \ar[rr]& &
	{\left(\begin{smallmatrix}\Omega_{\La}(M)\\ P^0\end{smallmatrix}\right)_{i_0}}\ar[r]& 0. } \ \    $$
The almost split sequences in $(a)$ and $(f)$ imply that $\tau_{H}\left(\begin{smallmatrix}\Omega_{\La}(M)\\ 0\end{smallmatrix}\right)_{0}=\left(\begin{smallmatrix}P^1\\ P^0\end{smallmatrix}\right)_{i_0p_1}$ and the indecomposable object $\left(\begin{smallmatrix}\Omega_{\La}(M)\\ P^0\end{smallmatrix}\right)_{i_0}$  is a direct summand of the middle term of the almost split sequence $\tilde{b}$ ending at $\left(\begin{smallmatrix}\Omega_{\La}(M)\\ 0\end{smallmatrix}\right)_{0}$. By comparing the  domains and codomains of the ending terms of  the almost split sequence $\tilde{b}$ (or the same dimension shifting), we see that the middle term of $\tilde{b}$ has exactly direct summands (up to isomorphism) $\left(\begin{smallmatrix}\Omega_{\La}(M)\\ P^0\end{smallmatrix}\right)_{i_0}$ and $\left(\begin{smallmatrix}P^1\\ 0\end{smallmatrix}\right)_{0}$, so we have $\tilde{b}$ is the same the sequence in $(b)$. This ends the  proof of the sequence in $(b)$ to be an almost split sequence in $\text{H}(\La)$. 

$(c)$ Applying \cite[Proposition 3.4]{H3}, we get  $\tau_{H}\left(\begin{smallmatrix}P^0\\ P^1\end{smallmatrix}\right)_{i_0p_1}=\left(\begin{smallmatrix}0\\ \Omega_{\La}(M)\end{smallmatrix}\right)_{0}$. Indeed, $\text{Cok}(i_0p_1)=M$ and $\tau_{\La}M=\Omega_{\La}(M)$, and then we can apply the result to get the aforementioned claim. Also, by the almost split sequence $(f)$ we see that $\left(\begin{smallmatrix}P^1\\ \Omega_{\La}(M)\end{smallmatrix}\right)_{0}$ is a direct summand of the almost split sequence $\tilde{c}$ ending at $\left(\begin{smallmatrix}P^1\\ P^0\end{smallmatrix}\right)_{i_0p_1}$. Next by comparing the domain and codomain, we see that the middle term of the $\tilde{c}$ has  exactly direct summands $\left(\begin{smallmatrix}P^1\\ \Omega_{\La}(M)\end{smallmatrix}\right)_{p_1}$ and $\left(\begin{smallmatrix}0\\ P^0\end{smallmatrix}\right)_{0}$, so $\tilde{c}$ is the same as the sequence in $(c)$. $(e)$ is a direct consequence of $(b)$ and $(c)$. 
\end{proof}
In the sequel,  we apply the above proposition several times for the case $\La=A_n=kZ_n/I^2$. Note that all indecomposable $A_n$-modules are either projective or simple. Take an arbitrary  simple module $S$. On can see that $\Omega^n_{A_n}(S)=S$ and $n$ is the least number with such a property. Consider the short exact sequences 
  $\epsilon_i: 0\rt K^i\rt P^i\rt K^{i-1}\rt 0$, for $ 0 \leqslant i \leqslant n-1$,  the induced short exact sequences by the minimal projective resolution of $S$,  here  $K^{-1}=S$ and  $K^{n-1}=\Omega^n_{\La}(S)$.
  
  Apply first Proposition \ref{Proposition 5.5} for the pair $(\epsilon_0, \epsilon_1)$ and of course in view of   Lemma \ref{Lemma 5.2} we get the following part of $\Gamma_{\rm{H}( \La)}$

\[
\xymatrix  @R=0.2cm  @C=0.2cm { 
		&&&0P^0\ar[dr]\ar@{.}[rr]& &P^10\ar[rd]&	\\   &	&0\Omega_{A_n}(S)\ar[dr]\ar[ru]\ar@{.}[rr]&&P^1P^0\ar[ru]\ar[dr]\ar@{.}[rr]&&\Omega_{A_n}(S)0\ar[dr]&\\  &
	&&P^1\Omega_{A_n}(S)\ar[ru]\ar@{.}[rr]\ar[dr]&&\Omega_{A_n}(S)P^0\ar[rd]\ar[ur]\ar[r]&P^0P^0\ar[r]&P^0S\\
	&	&\Omega^2_{\La}(S)0\ar[ru]\ar@{.}[rr]&&\Omega_{A_n}(S)\Omega_{A_n}(S)\ar[ur]\ar@{.}[rr]&&0S\ar[ur]&}
\]
Continuing the above procedure for the rest of the short exact sequences, then we reach to a part (as above) of $\Gamma_{\rm{H}(\La)}$ such that the vertices in the leftmost side are  $0\Omega^n_{A_n}(S)=0S$ and $\Omega^{n+1}_{A_n}(S)0=\Omega_{A_n}(S)0$. Therefore, we obtain entirely  $\Gamma_{\rm{H}(A_n)}$ by gluing the $n$ parts. The immediate consequence of the above construction is that $\text{H}(A_n)$ is of finite representation type, and consequently by \cite[Theorem 1.1]{AR2} the Auslander algebra of $A_n$ so is.

For instance, for the case $n=2$, the $\Gamma_{\text{H}(A_2)}$ is of the form:

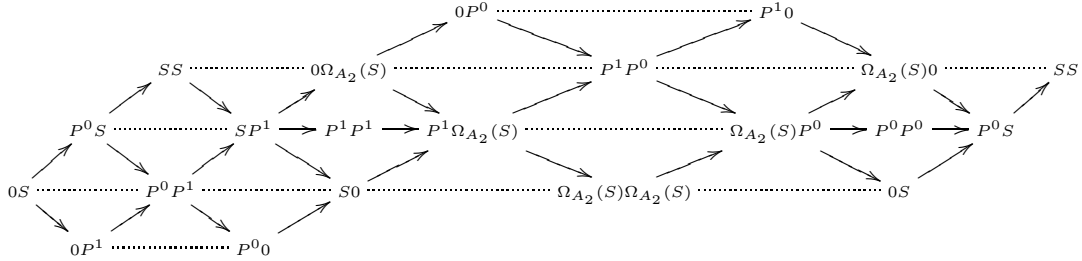
\[\tiny{
\xymatrix  @R=0.3cm  @C=0.3cm { 
&&	&&&0P^0\ar[dr]\ar@{.}[rr]& &P^10\ar[rd]&&	\\ &&SS\ar[rd]\ar@{.}[rr]  &	&0\Omega_{A_2}(S)\ar[dr]\ar[ru]\ar@{.}[rr]&&P^1P^0\ar[ru]\ar[dr]\ar@{.}[rr]&&\Omega_{A_2}(S)0\ar[dr]\ar@{.}[rr]&&SS\\ &P^0S\ar[rd]\ar[ru]\ar@{.}[rr]& &
SP^1\ar[ru]\ar[r]\ar[rd]	&P^1P^1\ar[r]&P^1\Omega_{A_2}(S)\ar[ru]\ar@{.}[rr]\ar[dr]&&\Omega_{A_2}(S)P^0\ar[rd]\ar[ur]\ar[r]&P^0P^0\ar[r]&P^0S\ar[ru]&\\0S\ar[rd]\ar[ru]\ar@{.}[rr]&&P^0P^1\ar[ru]\ar[rd]\ar@{.}[rr]	&	&S0\ar[ru]\ar@{.}[rr]&&\Omega_{A_2}(S)\Omega_{A_2}(S)\ar[ur]\ar@{.}[rr]&&0S\ar[ur]&&\\&0P^1\ar[ru]\ar@{.}[rr]&&P^00\ar[ru]&&&&&&}
}\]

Firstly,   if we delete the projective vertices and injective vertices (those vertices are corresponded to $\left(\begin{smallmatrix}P\\ P\end{smallmatrix}\right)_{1}$  or $\left(\begin{smallmatrix}0\\ P\end{smallmatrix}\right)_{0}$ or $\left(\begin{smallmatrix}P\\ 0\end{smallmatrix}\right)_{0}$ for some indecomposable projective $A_2$-module $P$) of $\Gamma_{\rm{H}(A_2)}$, denote by  $\Gamma^{ss}_{\rm{H}(A_2)}$ the resulting quiver, according to our above computation  we see that $\Gamma^{ss}_{\rm{H}(A_2)}=Z\mathbb{A}_3/(\tau^4)$.

Secondly, first consider  these  two facts as follows: $(1)$ in a straightforward way as we did in the proof of Proposition \ref{Prop 5.3}, all the almost split sequences in $\text{H}(A_2)$ have no ending terms of the form $\left(\begin{smallmatrix}M\\ M\end{smallmatrix}\right)_{1}$  or $\left(\begin{smallmatrix}M\\ 0\end{smallmatrix}\right)_{0}$, for some indecomposable module $M$, their image  under the functor $\Phi:\text{H}(A_2)\rt \mmod (\mmod A_2)$ are again almost split in $\mmod (\mmod A_2)$, or in $\mmod B$, where $B$ is the Auslander algebra of $A_2$. See also \cite[Proposition 5.7]{HMa} for a proof of the mentioned fact in a general situation.
  $(2)$  projective-injective functors in $\mmod (\mmod A_2)$ are those representable functors of the form $\Hom_{A_2}(-, P)$, where $P$ is a projective module in $\mmod A_2.$ Following these two aforementioned facts and our computation for $\Gamma_{\rm{H}(A_2)}$ we get the stable Auslander-Reiten quiver $\Gamma^s_B$, obtained by deleting projective-injective  vertices,  is equal to the extended Dynkin quiver of type  $\tilde{\mathbb{A}}_{8}$ with clockwise orientation, as presented in below, 

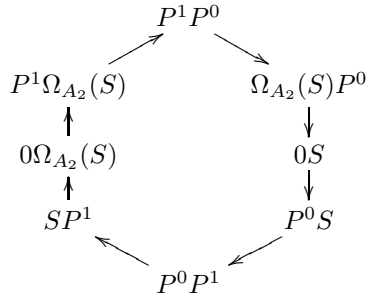
\[
\xymatrix@R1em@C.5em{&P^1P^0\ar@{<-}[dl]\ar@{->}[dr]\\ P^1\Omega_{A_2}(S)\ar@{<-}[d]&&
	\Omega_{A_2}(S)P^0\ar@{->}[d]\\0\Omega_{A_2}(S) &&0S\ar[d]\\ SP^1\ar@{<-}[dr]\ar[u]&&P^0S\ar@{->}[dl]\\ &P^0P^1}\ \ \ 
\]

\section{1-Gorenstein $\Omega_{\mathcal{G}}$-algebras}
In this section we mainly study $\Omega_{\CG}$-algebras which are also 1-Gorenstein. Recall that  an algebra $\La$ is $1$-Gorenstein if the injective dimension of $\La$ as left or right module (in this case one side is enough) is at most one. In the last result, we show the $\Omega_{\La}$-algebras (not necessarily 1-Gorenstein) are a good source to produce CM-finite algebras.

	In the following we shall give a characterization of $1$-Gorenstein $\Omega_{\CG}$-algebras in terms of  radical  of projective modules.
\begin{proposition}\label{RadGorp}
	The  following conditions are equivalent
	\begin{itemize}
		\item[$(i)$] $\La$ is a $1$-Gorenstein $\Omega_{\CG}$-algebra.  
		\item[$(ii)$]$\text{rad} \La\oplus \La$ generates $\text{Gprj} \mbox{-} \La$, i.e., $\text{Gprj} \mbox{-} \La=\text{add}(\text{rad} \La\oplus \La)$ .
		
	\end{itemize}		 
\end{proposition}
\begin{proof}
	$(i)\Rightarrow (ii)$.  $1$-Gorensteiness of $\La$ 
	implies $\Omega^1_{\La}(\mmod \La)=\rm{Gprj} \mbox{-} \La$,  in particular, $\text{rad}\La$ is a Gorenstein projective module. We observe by Lemma \ref{Radical} that for each indecomposable projective $Q$, $\rm{rad}(Q)\hookrightarrow Q$ is a  minimal right  almost split morphism in  $\rm{Gprj} \mbox{-} \La$.  Let $G$ be an indecomposable non-projective Gorenstein projective module. There exists an indecomposable non-projective Gorenstein projective module  $G'$ such that $\Omega_{\La}(G')=G$. Then we have the almost split  sequence $0 \rt  G  \st{f} \rt  P \rt G' \rt 0$ in $\rm{Gprj}\mbox{-}\La,$ by using  the assumption of being $\Omega_{\CG}$-algebra of $\La$.
	
	 Assume $Q$ is an indecomposable direct summand of $P$. From the left minimal almost split morphism $f$ we get an irreducible morphism from $G$ to $Q$. As we mentioned in the above there exits the  right minimal almost split morphism $\text{rad}(Q)\hookrightarrow Q$. Hence $G$ must be a direct summand of $\text{rad}(Q)$, so a direct summand of $\text{rad}\La$, as required.

	    $(ii)\Rightarrow (i).$   Since $\text{rad}\La$ is a Gorenstein projective module, we obtain this fact that the simple modules have Gorenstein projective dimension at most one. Then by induction on the length of modules, we can prove that the global Gorenstein projective dimension of $\mmod\La$ is at most one, equivalently, $\La$ is $1$-Gorenstein. To prove $\La$ to be an $\Omega_{\CG}$-algebra, we first claim that there are no irreducible morphisms between two  indecomposable   non-projective Gorenstein projective  modules. Suppose, contrary to our claim, that we would have an irreducible morphism $f:G \rt G'$ with $G$ and $G'$ being indecomposable  non-projective Gorenstein projective  modules. As $\rm{Gprj} \mbox{-} \La$ is closed under submodule, in the same we as the module category case we deduce that   $f$ would be either a monomorphism or an epimorphism. We first assume that $f$ is an epimorphism. Let $0\rt  \tau_{\CG}G'\rt E\st{g}\rt G'\rt 0$ be an almost split sequence in $\rm{Gprj} \mbox{-} \La$.  Clearly, $\tau_{\CG}G'$ is an indecomposable non-projective Gorenstein projective module, by our assumption, there is an indecomposable  projective module  $P$ such that 
	$\tau_{\CG}(G')$ is a direct summand of $\rm{rad}(P).$ Since the inclusion $\rad(P)\hookrightarrow P$ is a right  minimal  
	almost split  morphism in $\text{Gprj}\mbox{-}\La$ (Lemma \ref{Radical}),  it implies that $P$ must be a direct summand of $E$. Thus we have the irreducible morphism $g\mid:P \rt G'$, and by the projectivity of $P$ and $f$ to be an epimorphism there is $h:P\rt G$ such that $f\circ h=g$. As $g$ is an irreducible morphism, then either  $h$ is a split monomorphism or $f$ a split epimorphism, but both are  impossible, as modules $G, G'$  are indecomposable non-projective  and but $P$ is indecomposable projective. In the other case, if $f$ is a monomorphism.  For this case we apply  the duality $(-)^*:\rm{Gprj} \mbox{-} \La \rt \rm{Gprj} \mbox{-} \La^{\rm{op}}$. In fact, $f^*$ is an epimorphism in  $\rm{Gprj} \mbox{-} \La^{\rm{op}}$	and then   follows  the first case.
	
	The  established claim gives us that for any indecomposable  non-projective Gorenstein projective indecomposable $X$, if $g:Y \rt X $ is a right minimal almost split morphism in $\rm{Gprj} \mbox{-} \La$, then $Y$ must be  a projective module. Note that by $(ii)$ the subcategory $\rm{Gprj}\mbox{-}\La$ is of finite representation  type, and so it has almost split sequences. Let an indecomposable non-projective representable functor $(-, \underline{G})$ is given.  We may assume that $G$ is indecomposable non-projective in $\rm{Gprj} \mbox{-} \La$. As we said there is an almost split sequence $\la: 0 \rt \tau_{\CG}G\rt E\rt G\rt 0$ in $\rm{Gprj} \mbox{-}\La.$ We know from our claim that $E$ is projective. Since $E$ is projective after applying the Yoneda functor on $\la$ we get the following sequence in $\mmod (\text{Gprj}\mbox{-}\La)$
	$$ 0 \rt (-, \tau_{\CG}G)\rt (-, E)\rt (-, G)\rt (-, \underline{G})\rt 0.$$ 
	As $\la$ is almost split in $\rm{Gprj} \mbox{-}\La$, the above sequence implies that $(-, \underline{G})$ is simple in $\mmod (\rm{Gprj} \mbox{-}\La).$  Consequently, as any simple functor in $\rm{mod}\mbox{-}(\underline{\rm{Gprj}}\mbox{-} \La)$ arising in this way, induced by an almost split sequence,  see e.g. \cite[Chapter 2]{Au3}, we get it a semisimple abelian category. So $\La$ is  an $\Omega_{\CG}$-algebra. We are done.
\end{proof}

In the following,  some examples of $1$-Gorenstein $\Omega_{\CG}$-algebras are given.	
\begin{example}
	\begin{itemize}
		\item [$(i)$]Clearly, hereditary algebras are $1$-Gorenstein $\Omega_{\CG}$-algebras. So $1$-Gorenstein $\Omega_{\CG}$-algebras  can be considered as a generalization of hereditary algebras. 
		\item[$(ii)$] Recently due to Ming Lu and Bin Zhu in \cite{LZ} a criteria was given for which monomial algebras are $1$-Gorenstein algebras. Hence by having in hand such criteria we can search among quadratic monomial algebras to find $1$-Gorenstein $\Omega_{\CG}$-algebras. In particular, we can specialize on gentle algebras to find which of them are $1$-Gorenstein, as studied in \cite{CL}.
		\item[$(iii)$] The cluster-tilted algebras  are defined in \cite{BMR3} and \cite{BMD4} are an important class of $1$-Gorenstein algebras. So among cluster-tilted algebras we can find some examples of $1$-Gorenstein $\Omega_{\CG}$-algebras (see below). For example,  the cluster-tilted algebras of type $A$ since are gentle algebra, so in this case we are dealing with $1$-Gorenstein $\Omega_{\CG}$-algebras.  For the  other types $D$ and $E$, in \cite{CGL} the singularity categories of cluster-tilted algebras are described by the stable categories of some self-injective algebras. In particular, those cluster-tilted algebras of type $D$ and $E$  which are singularity equivalent to the self-injective Nakayama algebra $\La(3, 2)$,  the Nakayama algebra with cycle quiver with $3$ vertices modulo the ideal generated by the paths of length $2$, are other examples of $\Omega_{\CG}$-algebras.  The following cluster-tilted algebra  given by the quiver 
		
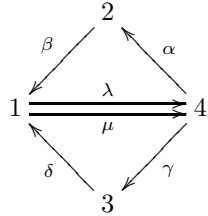
\begin{equation*} 
		\xymatrix{& 2 \ar[ld]_{\beta} \\
			1\ar@<2pt>[rr]^{\lambda}\ar@<-2pt>[rr]_{\mu}& & 4 \ar[lu]_{\alpha}\ar[ld]^{\gamma}\\ 
			&3\ar[lu]^{\delta}} 
		\end{equation*}	
		
		bound by the quadratic monomial relations $\alpha \beta=0, \  \gamma \delta=0, \ \delta \lambda=0, \  \lambda \gamma=0, \ \beta \mu=0,$ and $\mu \alpha=0 $ is also a $1$-Gorenstein $\Omega_{\CG}$-algebra.
	\end{itemize} 
\end{example}

There is a similar result of the following theorem in \cite[Theorem 4.4]{CL2}, where the authors proved that over a gentle algebra $\La$: the Cohen-Macaulay Auslander algebra $\rm{Aus}(\rm{Gprj} \mbox{-} \La)$ of $\La$ is representation-finite if and only if   so is $\La$.

To prove we need some preparation from \cite{AHKa} as follows. 
Let $(\mathbb{A}_3, J)$ be the quiver $\mathbb{A}_3: v_2 \st{a}\rt v_1 \st{b}\rt  v_0$ with relation $J$ generated by $ab$. By definition, a representation $M$ of $(\mathbb{A}_3, J)$ over $\La$ is a diagram
\[M: (\ M_2 \st{f_2}\lrt M_1 \st{f_1}\lrt M_0)\]
of $\La$-modules and $\La$-homomorphisms such that $f_1 f_2=0$, i.e. a sequence $M_2 \st{f_2}\lrt M_1 \st{f_1}\lrt M_0$ in $\mmod \La$. A morphism between the representations $M$ and $N$ is a triplet $\alpha=(\alpha_2, \alpha_1, \alpha_0)$ of $\La$-homomorphisms such that the diagram
\[\xymatrix{M_2 \ar[r]^{f_2} \ar[d]^{\alpha_2} & M_1 \ar[r]^{f_1} \ar[d]^{\alpha_1} & M_0 \ar[d]^{\alpha_0} \\
	N_2 \ar[r]^{g_2} & N_1 \ar[r]^{g_1} & N_0 }\]
is commutative. Denote by $\text{rep}(\mathbb{A}_3, J, \La)$ the category of all representations of $(\mathbb{A}_3, J)$ over $\mmod \La.$
Consider the following  subcategory  of $\text{rep}(\mathbb{A}_3, J, \La)$:
\[
\mathscr{L}(\text{Gprj}\mbox{-}\La)  = \{(G_2 \st{f_2}\hookrightarrow G_1 \st{f_1}\rt G_0) \mid  G_0, G_1, G_2 \in \text{Gprj}\mbox{-}\La \ {\rm and \ } 0 \rt G_2\st{f_2}\rt G_1 \st{f_1} \rt G_0 \ {\rm is \ exact}\}. \]
Following \cite[Section 2]{AHKa}, we define the functor $\mathscr{F}:\mathscr{L}(\text{Gprj}\mbox{-}\La) \rt \mmod (\text{Gprj}\mbox{-}\La)$ by sending a representation  $G:( G_2 \st{f_2}\hookrightarrow G_1 \st{f_1}\rt G_0) $ in $\mathscr{L}(\text{Gprj}\mbox{-}\La)$ to the functor $F_G$ in $\mmod (\text{Gprj}\mbox{-}\La)$ lying in the following exact sequence  in $\mmod (\text{Gprj}\mbox{-}\La)$ (obtained by applying the Yoneda functor over $G$, if we consider it as a left exact sequence in $\mmod \La$)
$$0 \rt (-, G_2)\rt (-, G_1)\rt (-, G_0)\rt F_G\rt 0.$$
It is proved in \cite[Proposition 2.2]{AHKa} over 1-Gorenstein algebra $\La$, the functor $\mathscr{F}$ is full, faithful and objective. Hence, by \cite[Appendix]{RZ}, we have the equivalence $\mathscr{L}(\text{Gprj}\mbox{-}\La)/<\mathcal{K}>\simeq \mmod (\text{Gprj}\mbox{-}\La)$, where $<\mathcal{K}>$ is the ideal of  $\mathscr{L}(\text{Gprj}\mbox{-}\La)$ generated by all the kernel objects of the functor $\mathscr{F}$. Note that by the definition one can see an indecomposable kernel object  is isomorphic to either representation $(0\rt X\st{1}\rt X)$ or $(X\st{1}\rt X\rt 0)$ for some indecomposable module $X$ in $\text{Gprj}\mbox{-}\La$.

 An important application from the equivalence which we will use in the next theorem is: $\mathscr{L}(\text{Gprj}\mbox{-}\La))$ is of finite representation type if and only if so is $\mmod (\text{Gprj}\mbox{-}\La)$, whenever $\La$ is assumed to be 1-Gorenstein.

\begin{theorem}\label{Theorem 1}
	Assume that $\La$ is a $1$-Gorenstein $\Omega_{\CG}$-algebra. Then  the Cohen-Macaulay Auslander algebra of  $\La$ is representation-finite if and only if  so is $\La$. 
\end{theorem}
\begin{proof}
	Specializing \cite[Lemma 3.3]{AHK} to the subcategory $\rm{Gprj} \mbox{-} \La$ provides
	 us the fully faithful functor $\vartheta_{\rho}:\mmod \La\rt \mmod (\rm{Gprj} \mbox{-} \La)$. The embedding follows the \textquotedblleft only if \textquotedblright  part.  
	
	For the \textquotedblleft if\textquotedblright  part assume that $\La$ is representation-finite.  Based on the discussion given in the lines before the theorem,  we observe that the subcategory $\mathscr{L}(\text{Gprj}\mbox{-}\La)$ of $\text{rep}(\mathbb{A}_3, J, \La)$ is of finite representation type if and only if $\mmod (\rm{Gprj} \mbox{-} \La)$ so is. Thus it is enough to show that $\mathscr{L}(\text{Gprj}\mbox{-}\La)$ is of   finite representation type. To do this, we divide the proof into three cases as follows. 
	
	 Denote  first by $\mathcal{D}$ the full additive subcategory of $\mathscr{L}(\text{Gprj}\mbox{-}\La)$ consisting of all indecomposable objects are isomorphic to   either $(G \st{1}\rt G \rt 0)$ or $(0 \rt G \st{1} \rt G)$ with indecomposable object $G$ in $\rm{Gprj}\mbox{-}\La$. Since $\mathcal{D}$   has only finitely many indecomposable representations up to isomorphism, as $\La$ is CM-finite,  hence we may only consider the indecomposable representations in $\mathscr{L}(\text{Gprj}\mbox{-}\La)$  not in $\mathcal{D}$ throughout  the proof.
	
	{\bf Case 1:} We will show  in this step that there is only
	a finite number of  indecomposable representations, up to
	isomorphism, of $\mathscr{L}(\text{Gprj}\mbox{-}\La)$  of the form  $(0\rt G \st{f}\rt P)$ with $P$ projective module and $f$  an injection. Let $\mathcal{D}_1$ denote the subcategory of $\mathscr{L}(\text{Gprj}\mbox{-}\La)$  which consists of all representations $(0\rt X\st{g}\rt Q)$ with $X \in \text{Gprj}\mbox{-}\La, Q\in \text{prj}\mbox{-}\La$ and $g$ an injection. Indeed, $\mathcal{D}_1$ is generated by the indecomposable representations as
	 we want to show in this step  that there are finitely many isomorphism classes of them.  We  define a  functor $\mathscr{T}:\mathcal{D}_1 \rt \mmod \La$ by sending $(0\rt X \st{g} \rt Q)$ to the $\text{Cok}(f),$ which is a  full and dense functor. Note that since $\La$ is
	 $1$-Gorenstein then for each module $M$ in $\mmod \La$ we have a short exact sequence $0 \rt G \rt P \rt M  \rt 0$ with $P$ projective  and $G$ Gorenstein projective. This shows
	  that the functor $\mathscr{T}$ is dense. Also by the projectivity of the last terms  of the representations in $\mathcal{D}_1$ we can prove that	the functor $\mathscr{T}$ is full. Finally, it is not difficult to see that the kernel of $\mathscr{T}$ consists of all morphisms in $\mathcal{D}_1$ factor through representations of the form $(0\rt Q\st{1}\rt Q)$ with $Q$ a projective module; denote by $\mathcal{K}_1$ the subcategory of $\mathscr{L}(\text{Gprj}\mbox{-}\La)$ consisting of all such  forms. So the functor is objective as well. Thus, by \cite[Appendix]{RZ},  the functor $\mathscr{T}$ induces an equivalence between the quotient (additive) category $\mathcal{D}_1/<\mathcal{K}_1>$ and $\mmod \La.$ Therefore, by the induced equivalence and using our assumption of $\La$ being representation-finite, we deduce that $\mathcal{D}_1$ so is, as desired.

	{\bf Case 2:} In this step we will prove that there is a finite number of	indecomposable representations up to isomorphism in   $\mathscr{L}(\text{Gprj}\mbox{-}\La)$  of  the form $(0\rt G \st{f}\rt G')$ with only $f$  to be an  injection and not any   restriction  on $G'$.
	Take an indecomposable object $(0 \rt G \st{f}\rt G')$ of  such a stated form. Since $G'$ is 
	a Gorenstein projective module then it can be embedded  into  a projective
	module, namely $i: G'\rt P.$ Now we have the representation $(0 \rt G \st{i\circ f} \rt P)$ which
	lies in $\mathcal{D}_1$. Considering $(0 \rt G \st{i\circ f} \rt P)$ as an object in the Krull-Schmidt category $\rm{rep}(\mathbb{A}_3, J, \La)$, then we can decompose it as $$(\dagger) \ \ \ \   \ \ \  (0 \rt G \st{i\circ f} \rt P)=\oplus(0 \rt G_i \st{f_i}\rt  P_i),$$ where the $P_i$ must be projective modules. Since $\mathcal{D}_1$ is clearly closed under direct summands, then the  $(0 \rt G_i \st{f_i}\rt  P_i)$ are indecomposable objects  in $\mathcal{D}_1.$ By the  decomposition $(\dagger)$, we also  obtain the following decomposition $$(0 \rt G \st{ f} \rt G')=\oplus(0 \rt G_i \st{f_i}\rt  \text{Im}(f_i)).$$ Clearly  $\text{Im}(f_i)$ are in $\rm{Gprj} \mbox{-} \La$ as $\La$ is a $1$-Gorenstein. On the other hand, since $(0 \rt G \st{f}\rt G')$ is indecomposable then there is some $j$ such that $(0 \rt G \st{f}\rt G')=(0 \rt G_j \st{f_j}\rt  \text{Im}(f_j))$. As we have seen  any indecomposable representation in Case 2 can be uniquely  determined  by  an  indecomposable object in $\mathcal{D}_1.$ But by  Case 1 , $\mathcal{D}_1$ is of finite representation  type, so this completes the proof of this step.
	
	{\bf Case 3:} Let $A=(G_1 \st{f}\rt G_2 \st{g}\rt G_3)$ be an
	indecomposable representation in $\mathscr{L}(\text{Gprj}\mbox{-}\La)$. Without of loss generality we may assume that $A$ is the form neither $(0\rt G\st{1}\rt G)$ nor $(G\st{1}\rt G\rt 0)$. If we consider the representation $A$ as a left exact sequence in $\mmod \La$, then  we have the following diagram		
	\[ \xymatrix{
		0 \ar[r] & G_1 \ar[r]^{f} & G_2 \ar[rd]^{e} \ar[rr]^{g}
		&
		&G_3 \\
		& &	&G\ar[ru]^{m} }\]	
	in which $g=me$ is  an epi-mono factorization of $g$. Since $\La$ is $1$-Goresntein then $G$ belongs to $\rm{Gprj} \mbox{-} \La$. In view of Lemma \ref{Lemma1}, concerning the structure  of the short exact sequences in $\rm{Gprj}\mbox{-}\La$ over an $\Omega_{\CG}$-algebra, and  $X$ being indecomposable, we obtain $G_2$ a projective module and $e$ a projective cover. In fact, if the short exact sequence $0 \rt G_1\st{f}\rt G_2\st{e}\rt G\rt 0$ has a direct summand of the form     $(M \st{1} \rt M \rt 0)$ or $(0 \rt M \st{1} \rt M)$, then it contradicts the indecomposability of $A$.
	
	 Since $(0\rt G\st{m}\rt G_3)$ is a
	monomorphism then  we can   consider  it  as an object appearing in Case 2.  We      decompose the  representation   into  indecomposable representations as the following
	$$(0 \rt G \st{m} \rt G_3)=\oplus(0 \rt H_i \st{m_i}\rt M_i).$$ Let $p_i: P_i \rt H_i$ be a projective cover of $H_i$ for each $i.$  Because $ e: G_2\rt G$ is a projective cover of $ G=\oplus H_i$, by the above decomposition we have the equality, so  we may assume  $G_2=\oplus P_i $, due to uniqueness of  projective covers,  and may identify  $e$  with  a morphism  such that whose matrix  presentation is a diagonal  matrix  with  $p_i$ on  the $(i, i)$-th entry (and so $f$ with a matrix presentation with  $h_i=\text{Ker}(p_i)\rt P_i $ on the $(i, i)$-entry and $G_1=\oplus\text{ker}(p_i)$, i.e., $\text{ker}(p_i)=\Omega_{\La}(H_i)$).    All together,  we have that the following decomposition 
	$$ (G_1 \st{f}\rt G_2 \st{g}\rt G_3)=\oplus(\Omega(H_i) \st{h_i} \rt P_i \st{m_i  p_i} \lrt M_i),$$
 As $X$ is an indecomposable representation then it is isomorphic to  $(\Omega(H_j)\st{h_j}\rt  P_j \st{m_j \circ p_j} \lrt M_j)$ for some $j$.  As we have observed the representations $(\Omega(H_i) \st{h_i}\rt P_i \st{m_ip_i}\rt M_i)$ are constructed uniquely by the indecomposable representations $(0 \rt H_i\st{m_i}\rt M_i)$. But, from Case 2 we deduce that there is a finite number of choices of the representations $(0\rt H_i\st{m_i}\rt M_i)$ up to isomorphism. Se we get  we have finitely many choices up to isomorphism for the given indecomposable $A$. We are done.
\end{proof}

An immediate consequence of the above  theorem and Proposition \ref{RadGorp} is  the following:
\begin{corollary}
	Let $\La$ be a $1$-Gorenstein $\Omega_{\CG}$-algebra. If $\La$ is representation-finite, then $\rm{End}(\La\oplus \text{rad}\La)$ so is. 
\end{corollary}

It is interesting to see that whether Theorem \ref{Theorem 1} holds for a more general algebras such as  $1$-Gorenstein CM-finite algebra.

We end up this  paper by the following result which shows that the $\Omega_{\CG}$-algebras are good source to produce CM-finite algebra via path algebras.

We first briefly outline some background about the path algebras. Given a finite-dimensional algebra $\La$ over a field $k$, and a finite acyclic quiver $\CQ$. Let 
$$\La \CQ=k\CQ \otimes_k \La, $$
where $k\CQ$ is the path algebra of $\CQ$ over $k$. We call $\La \CQ$ the path algebra of a finite acyclic  quiver $\CQ$ over $\La$. As in the case of $\La=k$, $\mmod \La \CQ$ is equivalent to the category $\text{rep}(\CQ, \La)$ of
representations of $\CQ$  over $\La$. The notion of Gorenstein projective representations are defined analog with Gorenstein projective modules. We refer to \cite{LuZ} for more details. For instance if $\CQ=v_1 \rt \cdots \rt v_n$, the path algebra $\La \CQ$ is given by the lower  triangular
matrix algebra of $\La$:
$$T_n(\La)=
\left[ \begin{array}{cccc}

\La & 0 & \cdots & 0 \\
\La & \La & \cdots & 0 \\
\vdots & \vdots & \ddots & \vdots \\
\La & \La & \cdots & \La \\
\end{array} \right]     
$$
Moreover, a representation $X$ of $\CQ$  over $\La$ is a datum as described in  below: 
$$X=(X_1\st{f_1}\rt X_2\st{f_2}\rt \cdots \st{f_{n-1}}\rt X_n)$$
where $X_i$ is a $\La$-module and $f_i:X_i\rt X_{i+1}$ is a $\La$-homomorphism. A morphism from representation $X$ to representation $Y$ is a datum $(\alpha_i:X_i\rt X_{i+1})_{1\leqslant i \leqslant n-1}$ such that for each $1\leqslant i \leqslant n-1$
\[\xymatrix{X_i \ar[r]^{f_i} \ar[d]_{\alpha_2} & X_{i+1} \ar[d]^{\alpha_{i+1}} \\
	Y_i \ar[r]^{g_i} & Y_{i+1}  }\]
commutes.

\begin{proposition}\label{linearquiver}
	Let $\La$ be an $\Omega_{\CG}$-algebra. Let $\CQ$ be the following linear quiver with $n \geq 1$ vertices $$v_1 \rt \cdots \rt v_n.$$  Then the path algebra $\La \CQ$ (or $T_n(\La)$) is $\rm{CM}$-finite.
\end{proposition}
\begin{proof}
	Since the case $n=1$ is clear, we assume  $n\geq 2.$ By the local characterization given in \cite[Theorem 5.1]{LuZ} or \cite[the dual of Theorem 3.5.1]{EHS} for the Gorenstein projective representations,   we obtain for a given representation  $X=( X_1 \st{f_1} \rt \cdots \rt X_{n-1} \st{f_{n-1}} \lrt X_n)$ in $\rm{rep}(\CQ, \La)$: $X$ is a Gorenstein projective representation if and only if it satisfies the following conditions
	\begin{itemize}
		\item [$(1)$] For $1 \leq i \leq n$, $X_i$ are Gorenstein projective modules. 
		\item [$(2)$] For $1 \leq i \leq n-1,$ $\text{Cok}(f_i)$ are Gorenstein projective modules and $f_i$ are monomorphisms. 
	\end{itemize}
Assume  $X$ is an indecomposable Gorenstein projective representation  in $\text{rep}(\CQ, \La)$ then  we claim that it is isomorphism to the following indecomposable representation (which is obvious to be Gorenstein projective from the above characterization and also in the end of the proof we will show that it is indecomposable): Let   $G$  be an indecomposable Gorenstein projective  module and $1\leq i \leq j \leq n$, consider 
	$$(\dagger) \ \ \ \ \ \ Y_{[i, j, G]}=(0 \rt \cdots 0 \rt \Omega_{\La}(G) \st{1}\rt \Omega_{\La}(G) \cdots \rt \Omega_{\La}(G) \st{1}\rt \Omega(G) \st{l} \rt P \st{1}\rt \cdots \st{1}\rt P) $$
	where the first $G$ is settled in the $i$-th vertex and the last one  in the $j$-th vertex, the map $l$  attached to the arrow $v_j\rt v_{j+1}$ is the inclusion $\Omega_{\La}(G)\hookrightarrow P$. Set  $\text{Cok}(l):=G'$.  Because of being indecomposable  of $G$ and $G'$, we also deduce that the map $l:G\rt P$ is a minimal left $\text{prj}\mbox{-}\La$-approximation and  the induced map $ P \rt G'$ is a projective cover..

	Let $m$ be the least integer such that $X_m\neq 0.$ We prove the claim by the inverse induction on $m$. If $m=n$, then the case is clear. Assume $m<n.$ Denote by $X'$ the sub-representation of $X$ such that $X'_m=0$ and $X'_d=X_d$ for all $m < d \leqslant n$.   In view of Lemma \ref{Lemma1} and being indecomposable of $X$, the monomorphism  $f_m:X_m\rt X_{m+1}$ is isomorphic to one of the following cases:
	 $(1): 0\rt G, \ (2):G\st{1}\rt G, (3):\Omega_{\La}(G)\hookrightarrow P$. The first case is impossible as $X_m$ would be 0, a contradiction.  If the second case holds, then one can see easily $\End_{\La \CQ}(X)\simeq \End_{\La \CQ}(X')$. Applying the above characterization follows that $X'$ remains to be Gorenstein projective. Hence  $X'$ is an indecomposable Gorenstein projective  representation with $X'_m=0$. Our inductive assumption implies that $X' \simeq Y_{[m-1, j, G]}$  and $1\leqslant j\leqslant n$. Thus, $X\simeq Y_{[m, j, G]}$. If the case $(3)$ holds, then we show that $X \simeq Y_{[m, m+1, G]}$. If $m+1=n$, then there is nothing to prove. Assume $m+1<n$. Analogous to the above for the monomorphism $f_{m+1}:P\rt X_{m+2}$, we have three cases. The first and third cases are impossible. Just it  remains the second case. This implies we may  identify $f_{m+1}$ by the identity map of $P$. Continuing this inductive procedure leads to the desired form. Lastly,  we  complete the 	 induction proof.

	Now we prove  $Y_{[i, j, G]}$ defined in the above is indecomposable.	Assume an endomorphism $\alpha=(\alpha_i)^n_1$ of $Y:=Y_{[i, j, G]}$. As it satisfies the commutativity conditions, we obtain $\alpha_i=\alpha_d=\alpha_j$ for $i \leqslant d \leqslant j$ and $\alpha_{j+1}=\alpha_d=\alpha_n$ for $j+1\leqslant d \leqslant n.$ If $\alpha_i=\alpha_j$ is an automorphism, then, by applying this fact that $l$ is a minimal left $\text{prj}\mbox{-}\La$-approximation, $\alpha_{j+1}$ so is an automorphism. Hence $\alpha$ is an automorphism. If $\alpha_i=\alpha_j$ is not an automorphism, then it is nilpotent as $G$ is indecomposable.  Due to the commutative diagram assigned to the arrow $v_{j}\rt v_{j+1}$ we have the following commutative diagram
	$$\xymatrix{
		0 \ar[r] &G \ar[d]^{\alpha_j} \ar[r]^l & P \ar[d]^{\alpha_{j+1}}
		\ar[r] &G' \ar[d]^{\beta}\ar[r] & 0&\\
		0 \ar[r] & G \ar[r]^{l} & P
		\ar[r] & G'\ar[r] &0.& \\ 	} 	 $$
	But the induced map $\beta$ is not an automorphism. Otherwise, as $P\rt \text{Cok}(l)=G'$ is a projective cover, we get $\alpha_{j+1}$ so is. Then by the above diagram we obtain $\alpha_j$ is an automorphism, a contradiction. Hence, $\beta$ is nilpotent, as $G'$ is indecomposable.	
\end{proof}

\begin{remark}
	By use of \cite[Corollary 7.6]{ABHV} we observe that over an Gorenstein algebra $\La$ Proposition \ref{linearquiver} is independent of the orientation of the quiver $\CQ$. Hence in this way we can produce more CM-finite algebras.  In general, we believe directly, as we did in Proposition \ref{linearquiver}, by using local characterization of Gorenstein projective representations given in \cite[Theorem 5.1]{LuZ} for a finite acyclic quiver and even more by the similar one given in \cite{LZh2} for acyclic quivers with monomial relations, one can generate  some  more CM-finite algebras  by taking (quotient algebra of) path algebras of $\Omega_{\CG}$-algebras. 	 
\end{remark}

\section*{Acknowledgments}

This work is supported by a grant from the University of Isfahan. Some parts of the results, especially the notion of $\Omega_{\CG}$-algebras, of the present paper were first appeared in \cite{H4}.

\end{document}